\documentclass[12pt]{amsart}
\usepackage{amscd,amssymb,amsthm,amsmath,amssymb,textcomp,multirow,rotating,longtable,mathrsfs,fancyhdr,textgreek,mathtools,wasysym,pdflscape,imakeidx,makecell}
\usepackage[matrix,arrow,curve]{xy}

\newtheorem{theorem}[equation]{Theorem}

\newtheorem{proposition}[equation]{Proposition}
\newtheorem{lemma}[equation]{Lemma}
\newtheorem{corollary}[equation]{Corollary}
\newtheorem*{corollary*}{Corollary}
\newtheorem*{maincorollary*}{Main Corollary}

\newtheorem*{conjecture*}{Conjecture}

\newtheorem*{problem*}{Problem}
\newtheorem*{calabiproblem*}{Calabi Problem}

\newtheorem*{theorem*}{Theorem}
\newtheorem*{maintheorem*}{Main Theorem}

\theoremstyle{definition}

\newtheorem*{example*}{Example}
\newtheorem{definition}[equation]{Definition}

\theoremstyle{remark}
\newtheorem{remark}[equation]{Remark}
\newtheorem*{remark*}{Remark}

\makeatletter\@addtoreset{equation}{section} \makeatother

\newcommand{\mumu}{\boldsymbol{\mu}}

\makeindex

\title{K-polystability of two smooth Fano threefolds}

\author{Ivan Cheltsov and Hendrik S\"u\ss}

\address{\emph{Ivan Cheltsov}
\newline
\textnormal{University of Edinburgh,  Edinburgh, Scotland}
\newline
\textnormal{\texttt{I.Cheltsov@ed.ac.uk}}}

\address{\emph{Hendrik S\"u\ss}
\newline
\textnormal{University of Manchester, Manchester, England}
\newline
\textnormal{\texttt{suess@sdf-eu.org}}}

\begin{document}

\begin{abstract}
We give new proofs of the~K-polystability of two smooth Fano threefolds.
One of them is a~smooth divisor in $\mathbb{P}^1\times\mathbb{P}^1\times\mathbb{P}^2$ of degree $(1,1,1)$, which is unique up to isomorphism.
Another one is the~blow up of the complete intersection
$$
\Big\{x_0x_3+x_1x_4+x_2x_5=x_0^2+\omega x_1^2+\omega^2x_2^2+\big(x_3^2+\omega x_4^2+\omega^2x_5^2\big)+\big(x_0x_3+\omega x_1x_4+\omega^2x_2x_5\big)\Big\}\subset\mathbb{P}^5
$$
in the conic cut out by $x_0=x_1=x_2=0$, where $\omega$ is a~primitive cube root of unity.
\end{abstract}

\maketitle

\section{Introduction}
\label{section:intro}

Let $X$ be a~smooth Fano threefold. Then $X$ is contained in one of $105$ families, which are explicitly described in \cite{IsPr99},
These families are labeled as \textnumero 1.1, \textnumero 1.2, $\ldots$, \textnumero 9.1, \textnumero 10.1,
and members of each family can be parametrized by an irreducible rational variety.

\begin{theorem}[{\cite{ACCFKMGSSV}}]
\label{theorem:ACCFKMGSSV}
Suppose that $X$ is a~general member of the~family \textnumero $\mathscr{N}$. Then
$$
X\ \text{is K-polystable} \iff\mathscr{N}\not\in\left\{\aligned
&\ 2.23, 2.26 2.28, 2.30, 2.31, 2.33, 2.35, 2.36, 3.14, \\
& 3.16, 3.18,3.21, 3.22, 3.23, 3.24, 3.26, 3.28, 3.29, \\
&\  \ \ 3.30, 3.31, 4.5, 4.8, 4.9, 4.10, 4.11, 4.12, 5.2
\endaligned
\right\}.
$$
\end{theorem}

In the proof of this theorem, many explicitly given smooth Fano threefolds has been proven to be K-polystable.
Among them are the~two threefolds described in the~abstract.

Let $G$ be a~reductive subgroup in $\mathrm{Aut}(X)$, and let $f\colon\widetilde{X}\to X$ be a~$G$-equivariant birational morphism with smooth $\widetilde{X}$,
and let $E$ be any~$G$-invariant prime divisor in $\widetilde{X}$.
We say that $E$ is a~$G$-invariant prime divisor \emph{over} $X$, and we let $C_X(E)=f(E)$.
Then
$$
K_{\widetilde{X}}\sim f^*(K_X)+\sum_{i=1}^{n}a_iE_i
$$
where $E_1,\ldots,E_n$ are $f$-exceptional surfaces, and $a_1,\ldots,a_n$ are strictly positive integers.
If $E=E_i$ for some $i\in\{1,\ldots,n\}$, we let $A_X(E)=a_i+1$.
Otherwise, we let $A_X(E)=1$. The number $A_X(E)$ is known as the~log discrepancy of the~divisor $E$.
Then we let
$$
S_X(E)=\frac{1}{(-K_X)^n}\int_{0}^{\infty}\mathrm{vol}\big(f^*(-K_X)-xE\big)dx
$$
and $\beta(E)=A_X(E)-S_X(E)$. We have the~following result:

\begin{theorem}[{\cite{Fujita2019Crelle,Li,Zhuang}}]
\label{theorem:Fujita-Li}
The smooth Fano threefold $X$ is $K$-polystable if $\beta(F)>0$ for every $G$-invariant prime divisor $F$ over $X$.
\end{theorem}

Now, we let
$$
\alpha_G(X)=\mathrm{sup}\left\{\epsilon\in\mathbb{Q}\ \left|\ \aligned
&\text{the log pair}\ \left(X, \frac{\epsilon}{m}\mathcal{D}\right)\ \text{is log canonical for any}\ m\in\mathbb{Z}_{>0}\\
&\text{and every $G$-invariant linear subsystem}\ \mathcal{D}\subset\big|-mK_{X}\big|\\
\endaligned\right.\right\}.
$$
This number, known as the~global log canonical threshold \cite{CheltsovShramovUMN},
has been defined in \cite{Ti87} in a~different way.
But both definitions agree by \cite[Theorem~A.3]{CheltsovShramovUMN}.
If $G$ is finite, then
$$
\alpha_G(X)=\mathrm{sup}\left\{\epsilon\in\mathbb{Q}\ \left|\ \aligned
&\text{the log pair}\ \left(X, \epsilon D\right)\ \text{is log canonical for every}\\
&\text{ $G$-invariant effective $\mathbb{Q}$-divisor}\ D\sim_{\mathbb{Q}} -K_{X}\\
\endaligned\right.\right\}.
$$
by \cite[Lemma~1.4.1]{ACCFKMGSSV}. We have the~following result:

\begin{theorem}[\cite{Ti87,ACCFKMGSSV}]
\label{theorem:Tian}
If $\alpha_G(X)\geqslant\frac{3}{4}$, then $X$ is K-polystable.
\end{theorem}

In this short note, we give a~new proof of the~K-polystability of the~threefolds described in the~abstract
using  Theorems~\ref{theorem:Fujita-Li} and \ref{theorem:Tian}. This is done in Sections~\ref{section:3-17} and \ref{section:2-16}.

\section{Smooth divisor in $\mathbb{P}^1\times\mathbb{P}^1\times\mathbb{P}^2$ of degree $(1,1,1)$}
\label{section:3-17}

Let $X$ be the~unique smooth Fano threefold in the~family \textnumero 3.17.
Then $X$ is the~divisor
$$
\Big\{x_0y_0z_2+x_1y_1z_0=x_0y_1z_1+x_1y_0z_1\Big\}\subset\mathbb{P}^1\times\mathbb{P}^1\times\mathbb{P}^2,
$$
where $([x_0:x_1],[y_0:y_1],[z_0:z_1:z_2])$ are coordinates on $\mathbb{P}^1\times\mathbb{P}^1\times\mathbb{P}^2$.

Let $G=\mathrm{Aut}(X)$. Then $G\cong\mathrm{PGL}_{2}(\mathbb{C})\rtimes\mumu_2$,
where $\mumu_2$ is generated by an~involution $\iota$ that acts as
$$
\big([x_0:x_1],[y_0:y_1],[z_0:z_1:z_2]\big)\mapsto\big([y_0:y_1],[x_0:x_1],[z_0:z_1:z_2]\big).
$$
and $\mathrm{PGL}_{2}(\mathbb{C})$ acts on each factor via an~appropriate irreducible $\mathrm{SL}_{2}(\mathbb{C})$-representation.
More precisely, an~element $\left(\begin{smallmatrix}
    a & b\\
    c & d\\
  \end{smallmatrix}\right)\in\mathrm{PGL}_{2}(\mathbb{C})$ acts as follows:
\begin{multline*}
\big([x_0:x_1],[y_0:y_1],[z_0:z_1:z_2]\big)\mapsto
\big([ax_0+cx_1:bx_0+dx_1],[ay_0+cy_1:by_0+dy_1],\\
[a^2z_0+2acz_1+c^2z_2:abz_0+(ad+bc)z_1+cdz_2:b^2z_0+2bdz_1+d^2z_2]\big)
\end{multline*}

There are birational contractions $\pi_1\colon X\to\mathbb{P}^1\times\mathbb{P}^2$ and
$\pi_2\colon X\to\mathbb{P}^1\times\mathbb{P}^2$ that contracts smooth irreducible surfaces $E_1$ and $E_1$ to smooth curves $C_1$ and $C_2$ of bi-degrees $(1,2)$.
Moreover, there exists $\mathrm{PGL}_{2}(\mathbb{C})$-equivariant commutative diagram
$$
\xymatrix{
&&X\ar@{->}[lld]_{\pi_1}\ar@{->}[rrd]^{\pi_2}&&\\%
\mathbb{P}^1\times\mathbb{P}^2\ar@{->}[rrd]_{\mathrm{pr}_2}&&&&\mathbb{P}^1\times\mathbb{P}^2\ar@{->}[lld]^{\mathrm{pr}_2}\\%
&&\mathbb{P}^2&&}
$$
where $\mathrm{pr}_2$ is the~projection to the~second factor,
the $\mathrm{PGL}_{2}(\mathbb{C})$-action on $\mathbb{P}^2$ is~\mbox{faithful},
and $\mathrm{pr}_2(C_1)=\mathrm{pr}_2(C_2)$ is the~unique $\mathrm{PGL}_{2}(\mathbb{C})$-invariant conic, which is given by $z_0z_2-z_1^2=0$.

By \cite[Lemma~4.2.6]{ACCFKMGSSV}, the threefold $X$ is K-polystable. Let us give~an alternative proof of this assertion.

Let $\mathrm{pr}_1\colon\mathbb{P}^1\times\mathbb{P}^2\to\mathbb{P}^1$ be the~projection to the~first factor.
Using $\mathrm{pr}_1\circ\pi_1$ and~$\mathrm{pr}_1\circ\pi_2$,
we obtain a~$\mathrm{PGL}_{2}(\mathbb{C})$-equivariant $\mathbb{P}^1$-bundle $\phi\colon X\to\mathbb{P}^1\times\mathbb{P}^1$,
where the~$\mathrm{PGL}_{2}(\mathbb{C})$-action on the~surface $\mathbb{P}^1\times\mathbb{P}^1$ is diagonal.
Let $C=E_1\cap E_2$. Then $\phi(C)$ is a~diagonal curve.
Denote its preimage on $X$ by $R$. Then $C=R\cap E_1\cap E_2$ and
$$
-K_X\sim E_1+E_2+R.
$$
Let $H_1=(\mathrm{pr}_1\circ\pi_1)^*(\mathcal{O}_{\mathbb{P}^1}(1))$,
let $H_2=(\mathrm{pr}_1\circ\pi_2)^*(\mathcal{O}_{\mathbb{P}^1}(1))$ and
let $H_L=(\mathrm{pr}_2\circ\pi_2)^*(\mathcal{O}_{\mathbb{P}^2}(1))$.
Then $\mathrm{Pic}(X)=\langle H_1,H_L,E_1\rangle$, $E_2\sim 2H_L-E_1$, $R\sim H_1+H_2\sim 2H_1+H_L-E_1$ and
$$
-K_X \sim 2H_1+3H_L-E_1.
$$

Observe that the~curve $C$ and the~surface $R$ are the~only proper $G$-invariant irreducible subvarieties~in~$X$.
This easily implies that $\alpha_G(X)=\frac{2}{3}$, so that we cannot apply Theorem~\ref{theorem:Tian} to prove that $X$ is K-polystable.
Let us apply Theorem~\ref{theorem:Fujita-Li} instead.

Let $\eta\colon Y\to X$ be a~$G$-equivariant birational morphism,
let $D$ be a~prime $G$-invariant divisor in $Y$, let $t$ be a~non-negative real number, and let
$$
S_X(D,t)=\frac{1}{-K_X^3}\int_0^{t}\mathrm{vol}\big(\eta^*(-K_X)-xD\big)dx.
$$
Then we have $S_X(D)=S(D,\infty)$ and $\beta(D)=A_X(D)-S_X(D)$.
By Theorem~\ref{theorem:Fujita-Li}, to prove that $X$ is K-polystable it is enough to show that $\beta(D)>0$.
Let us first show this in the~case when $\eta$ is an~identify map:

\begin{lemma}
\label{lemma:3-17-beta-R}
One has $S_X(R)=\frac{4}{9}$ and $\beta(R)=\frac{5}{9}$.
\end{lemma}

\begin{proof}
Since $-K_X=E_1+E_2+R$, the~pseudoeffective threshold $\tau(E)$ is $1$, so that
\begin{multline*}
S_X(R)=\frac{1}{-K_X^3}\int_0^1(-K_X-xR)^3dx=\\
=\int_0^1 -R^3 x^3 + R^2(-K_X)x^2- R(-K_X)^2+(-K_X)^3\,dx=\\
=\frac{1}{36}\int_{0}^1 12x^2 - 48x + 36\,dx=\frac{4}{9}.
\end{multline*}
Since $A_X(R)=1$, we have $\beta(R)=\frac{5}{9}$.
\end{proof}

Let $f\colon\widetilde{X} \to X$ be the~blow-up of the~curve $C$, let $E$ be the~exceptional surface of $f$,
let~$\widetilde{R}$, $\widetilde{E}_1$, $\widetilde{E}_2$ be the~proper transforms on $\widetilde{X}$ of the~surfaces $R$, $E_1$, $E_2$, respectively.
Then
$$
\left\{\aligned
&\widetilde{E}_1 \sim f^*(E_1)-E,\\
&\widetilde{E}_2 \sim 2f^*(H_L)-f^*(E_2)-E, \\
&\widetilde{R} \sim f^*(2H_1+H_L-E_1)-E.
\endaligned
\right.
$$

\begin{lemma}
\label{lemma:3-17-beta-E}
One has $S_X(E)=\frac{11}{9}$ and $\beta(E)=\frac{7}{9}$. Moreover, if $0\leqslant t\leqslant 1$, then
$$
S_X(E,t)=\frac{1}{36}\int_0^t(36 - 18t + 4x^3)dx = \frac{1}{36}t^4 - \frac{1}{4}t^3 + t.
$$
\end{lemma}

\begin{proof}
We have
$$
f^*(-K_X)-xE\sim f^*(R+E_1+E_2)-xE \sim \widetilde{R}+\widetilde{E}_1+\widetilde{E}_2+(3-x)E.
$$
so that $\tau(E)=3$. If $0\leqslant x\leqslant 1$, then $f^*(-K_X)-xE$ is nef.
Thus, if $x \in [0,1]$, then
\begin{multline*}
\mathrm{vol}(f^*(-K_X)-xE)=\Big(f^*(-K_X)-xE\Big)^3=\\
=f^*(-K_X)^3+3x^2f^*(-K_X)E^2-x^3E^3=36-18x^2+4x^3.
\end{multline*}

If $3>x>1$, then both surfaces $\widetilde{E}_1$ and $\widetilde{E}_2$ lies in the~asymptotic base locus of the~big divisor $f^*(-K_X)-xE$.
Moreover, if $x\in[1,2]$, then its Zariski decomposition is
$$
f^*(-K_X)-xE\sim_{\mathbb{R}}\frac{1}{2}(x-1)\big(\widetilde{E}_1+\widetilde{E}_2\big)+\underbrace{\Big(f^*(-K_X)-xE-\frac{1}{2}(x-1)\big(\widetilde{E}_1+\widetilde{E}_2\big)\Big)}_{\text{nef part}}.
$$
Thus, if $x\in[1,2]$, then we have
$$
\mathrm{vol}(f^*(-K_X)+xE)=\Big(f^*(-K_X)-xE-\frac{1}{2}(x-1)\big(\widetilde{E}_1+\widetilde{E}_2\big)\Big)^3=6x^2-36x+52.
$$
If $x\in (2,3)$, then  the~nef part of the~Zariski decomposition of $f^*(-K_X)-xE$ is
$$
f^*(-K_X)-xE-\frac{1}{2}(x-1)\big(\widetilde{E}_1+\widetilde{E}_2\big)+(x-2)\widetilde{R}.
$$
Thus, if $x\in[2,3]$, then
$$
\mathrm{vol}(f^*(-K_X)-xE)=\Big(f^*(-K_X)-xE-\frac{1}{2}(x-1)\big(\widetilde{E}_1+\widetilde{E}_2\big)+(x-2)\widetilde{R}\Big)^3=4(3-x)^3.
$$

Summarizing and integrating, we see that
$$
S_X(E)=\frac{1}{36}\int_{0}^{1}(36-18x^2+4x^3)dx+\frac{1}{36}\int_{1}^{2}(6x^2-36x+52)dx+\frac{1}{36}\int_{2}^{3}4(3-x)^3dx=\frac{11}{9},
$$
which gives $\beta(E)=\frac{7}{9}$, because $A_X(E)=2$. Similarly, we compute $S_X(E,t)$.
\end{proof}

The~action of the~group $G$ lift to the~threefold $\widetilde{X}$,
and $E\cap\widetilde{R}$ is a~$G$-invariant irreducible  curve,
which is contained in the~pencil $|\widetilde{R}\vert_{E}|$.
Therefore, using \cite[Theorem~5.1]{Mabuchi}, we see that the~group $\mathrm{PGL}_2(\mathbb{C})$ must act trivially on the~fibers of the~natural projection $E\to C$.
Since the~curves $\widetilde{E}_1\vert_{E}$ and $\widetilde{E}_1\vert_{E}$ are swapped by $G$, we see conclude that
$|\widetilde{R}\vert_{E}|$ contains exactly two $G$-invariant curves: $E\cap\widetilde{R}$ and another curve, which we denote by $C^\prime$.

Now, let $g\colon \widehat{X}\to\widetilde{X}$ be the~blow up of the~curve $C^\prime$,
let $R^\prime$ be the~$f$-exceptional surface,
let $\widehat{E}_1$, $\widehat{E}_2$, $\widehat{E}$, $\widehat{R}$ be the~proper transforms on $\widehat{X}$ of the~surfaces $E_1$, $E_2$, $E$, $\widetilde{R}$, respectively.
Then we have
$$
(f\circ g)^*(-K_X)\sim_{\mathbb{R}}\widehat{E}_1+\widehat{E}_2+\widehat{R}+3\widehat{E}+3R^\prime,
$$
which implies that the~pseudoeffective threshold $\tau(R^\prime)=3$.
On the~other hand, we have

\begin{lemma}
\label{lemma:3-17-R-prime}
One has $\beta(R^\prime)\geqslant \frac{5}{9}$.
\end{lemma}

\begin{proof}
Let $x$ be a~non-negative real number such that $x<3$. Then $\widehat{E}$ lies in the~stable base locus of the~divisor $(f\circ g)^*(-K_X)-xF$,
and the~positive part of the~Zariski decomposition of this divisor  has the~following form:
$$
(f\circ g)^*(-K_X)-\frac{t}{2}\widehat{E}-xR^\prime-D
$$
for an~effective $\mathbb{R}$-divisor $D$.
Indeed, if $\ell$ is a~general fiber of the~projection $\widehat{E}\to C$, then
$$
\Big((f\circ g)^*(-K_X)-xR^\prime\Big)\cdot\ell=-x
$$
and $\widehat{E}\cdot\ell=-2$, which implies the~required assertion. Thus, we have
$$
S_X(F)\leqslant 2S_X(E)=\frac{22}{9},
$$
because $S_X(E)=\frac{11}{9}$ by Lemma~\ref{lemma:3-17-beta-E}. Then
$$
\beta(F)=A_X(F)-S_X(F)=3-S_X(F)\geqslant 3-\frac{22}{9}=\frac{5}{9}
$$
as required.
\end{proof}

The~action of the~group $G$ lifts to  $\widehat{X}$, and the~surfaces $R^\prime$, $\widehat{E}$ and $\widehat{R}$ are $G$-invariant.

\begin{remark}
\label{remark:3-17}
There exists the~following $G$-equivariant commutative diagram:
$$
\xymatrix{
&&\widetilde{X}\ar@{->}[dl]_f\ar@{->}[dr]^h&\widehat{X}\ar@{->}[l]_g\ar@{->}[d]^\upsilon\\
R\ar@{->}[d]\ar@{^{(}->}[r]&X\ar@{->}[d]^{\phi}&&\overline{X}\ar@{->}[d]^\psi\\%
C\ar@{^{(}->}[r]&\mathbb{P}^1\times\mathbb{P}^1\ar@{=}[rr]&&\mathbb{P}^1\times\mathbb{P}^1}
$$
where $h$ is the~contraction of the~surface $\widetilde{R}$, $\upsilon$ is the~contraction of the~surfaces $R^\prime$ and~$\widehat{R}$,
and $\psi$ is a~$\mathbb{P}^1$-bundle.
Moreover, one can show that $\overline{X}\cong\mathbb{P}(\mathcal{O}_{\mathbb{P}^1 \times \mathbb{P}^1}(2,0)\oplus\mathcal{O}_{\mathbb{P}^1 \times \mathbb{P}^1}(0,2))$,
so that there is an~involution $\sigma\in\mathrm{Aut}(\overline{X})$ such that $\sigma$ swaps the~curves $\upsilon(R^\prime)$ and $\upsilon(\widehat{R})$.
Then~$\sigma$ lifts to $\widehat{V}$ and swaps the~divisors $R^\prime$ and $\widehat{R}$.
\end{remark}

The threefold $\widetilde{X}$ contains two $G$-invariant irreducible curves: the~curves $E\cap\widetilde{R}$ and~$C^\prime$.
The threefold $\widehat{X}$ also contains just two $G$-invariant irreducible curves: $\widehat{E}\cap\widehat{R}$ and $\widehat{E}\cap R^\prime$,
which are swapped by the~involution $\sigma$ from Remark~\ref{remark:3-17}.
Blowing up one of the~curves, we obtain a~new threefold that contains exactly three $G$-invariant irreducible curves that can be described in a~very similar manner.
Now, iterating this process, we obtain infinitely many $G$-invariant prime divisors over $X$,
which can be described  using weighted blow ups.

\begin{definition}
\label{definition:weighted-blow-up}
Let $V$ be a~smooth threefold that contains two smooth irreducible distinct surfaces $A$ and $B$ that
intersect transversally along a~smooth irreducible curve~$Z$,
and let $\theta\colon U\to V$ be the~weighted blow up with weights $(a,b)$ of the~curve $Z$~with respect to the~local coordinates along $Z$ that are given by the~equations of the~surfaces~$A$~and~$B$,
and let $F$ be the~exceptional surface of the~weighted blow up $\theta$. Then
\begin{itemize}
\item the~morphism $\theta$ is said to be an~$(a,b)$-blowup between $A$ and $B$,
\item the~surface $F$ is said to be an~$(a,b)$-divisor between $A$ and $B$.
\end{itemize}
\end{definition}

Observe that $(1,1)$-blow up in this construction is the~usual blow up of the~intersection curve.
To proceed, we need the~following well-known result:

\begin{lemma}
\label{lemma:iterated-blowup-Fn}
In the~assumptions of Definition~\ref{definition:weighted-blow-up} and notations introduced in this definition, suppose  that $(a,b)=(1,1)$ and $Z\cong\mathbb{P}^1$.
Let $n=|\alpha-\beta|$, where $\alpha$ and $\beta$ be integers such~that
$$
Z^2=\left\{\aligned
&\alpha\ \text{on the~surface A},\\
&\beta\ \text{on the~surface B}.\\
\endaligned
\right.
$$
Denote by $\widetilde{A}$ and $\widetilde{B}$ the~proper transforms on $U$ of the~surfaces $A$ and $B$, respectively.
Then $F\cong\mathbb{F}_n$, the~surfaces $\widetilde{A}$ and $\widetilde{B}$ are disjoint,
$\widetilde{A}\vert_{E}$ and $\widetilde{B}\vert_{E}$ are sections of the~natural projection $F\to Z$ such that $(\widetilde{A}\vert_{E})^2=(\beta-\alpha)$ and $(\widetilde{B}\vert_{E})^2=(\alpha-\beta)$.
\end{lemma}

\begin{proof}
Left to the~reader.
\end{proof}

Now, we are ready to prove

\begin{lemma}
\label{lemma:3-17-weighted-blowups}
All $G$-invariant prime divisors over $X$ can be described as follows:
\begin{enumerate}
\item the~surfaces $R$, $E$ or $R^\prime$,

\item an~$(a,b)$-divisor between $E$ and $\widetilde{R}$,

\item an~$(a,b)$-divisor between $\widehat{E}$ and $R^\prime$.
\end{enumerate}
\end{lemma}

\begin{proof}
Let $F$ be a~$G$-invariant prime divisor over $X$ such that $F$ is different from $R$, $E$,~$R^\prime$.
Then its center on $\widetilde{X}$ is one of the~curves $E\cap\widetilde{R}$ or~$C^\prime$.
Keeping in mind Remark~\ref{remark:3-17}, we~may assume that its center on $\widetilde{X}$ is $E\cap\widetilde{R}$.
Let us show that $F$ is an~exceptional divisor of a~weighted blow up between the~surfaces $E$ and $\widetilde{R}$,

Let $V_0=X$ and $Z_0=E\cap\widetilde{R}$. Then there exists a~sequence of $G$-equivariant blow ups
$$
\xymatrix{
V_m\ar@{->}[rr]^{\theta_m}&&V_{m-1}\ar@{->}[rr]^{\theta_{m-1}}&&\cdots\ar@{->}[rr]^{\theta_2}&&V_1\ar@{->}[rr]^{\theta_1}&& V_0}
$$
such that $\theta_1$ is the~blow up of the~curve $Z_0$,
the surface $F$ is the~$\theta_m$-exceptional surface,
the morphism $\theta_k$ is a~blow up of a~$G$-invariant irreducible smooth curve $Z_{k-1}\subset V_{k-1}$
such that the~curve $Z_{k-1}$ is contained in the~$\theta_{k-1}$-exceptional surface provided that $k\geqslant 2$.

For every $k\in\{1,\ldots,m\}$, let $F_k$ be the~$\theta_k$-exceptional surface, so that we have $F=F_m$.
To prove that $F=F_m$ is an~exceptional divisor of a~weighted blow up between $E$ and $\widetilde{R}$,
it~sufficient to prove the~following assertion for every $k$:
\begin{itemize}
\item the~surface $F_k$ contains exactly two $\mathrm{PGL}_2(\mathbb{C})$-invariant irreducible curves,
\item the~two $\mathrm{PGL}_2(\mathbb{C})$-invariant irreducible curves in $F_k$ are disjoint,
\item if $\mathscr{C}$ is a~$\mathrm{PGL}_2(\mathbb{C})$-invariant irreducible curve in $F_k$, then $\mathscr{C}$ is cut out by the~strict transform of one of the~following surfaces:
\begin{itemize}
\item the~surface $F_r$ for some $r\in\{1,\ldots,m\}$ such that $r\ne k$,
\item the~surface $E$,
\item the~surface~$\widetilde{R}$.
\end{itemize}
\end{itemize}
Clearly, it is enough to prove this assertion only for $k=m$. Let us do this.

Let $F_0=E$ and $F_{-1}=\widetilde{R}$. For every $k\in\{-1,0,1,\ldots,m-1\}$, let $\overline{F}_k$ be the~proper transform
of the~surface $F_k$ on the~threefold $V_m$. We claim that
\begin{itemize}
\item[(i)] $F_m\cong\mathbb{F}_n$ for some $n>0$;
\item[(ii)] the~surface $F_m$ contains exactly two $\mathrm{PGL}_2(\mathbb{C})$-invariant irreducible curves,
\item[(iii)] the~two $\mathrm{PGL}_2(\mathbb{C})$-invariant irreducible curves in $F_m$ are disjoint,
\item[(iv)] if $\mathscr{C}$ is a~$\mathrm{PGL}_2(\mathbb{C})$-invariant irreducible curve in $F_m$, then $\mathscr{C}^2\in\{-n,n\}$,
\item[(v)] if $\mathscr{C}$ is a~$\mathrm{PGL}_2(\mathbb{C})$-invariant irreducible curve in $F_m$, then
$$
\mathscr{C}=F_m\cap\overline{F}_r
$$
for some $r\in\{-1,0,1,\ldots,m-1\}$ and the~following assertions hold:
\begin{itemize}
\item if $\mathscr{C}^2=n$ on the~surface $F_m$, then $\mathscr{C}^2\leqslant 0$ on the~surface  $\overline{F}_r$,
\item if $\mathscr{C}^2=-n$ on the~surface $F_m$, then $\mathscr{C}^2>0$ on the~surface  $\overline{F}_r$.
\end{itemize}
\end{itemize}
Let us prove this (stronger than we need) statement by induction on $m$.

Suppose that $m=1$.
We already know that $F_0=E\cong\mathbb{P}^1 \times \mathbb{P}^1$ and $F_{-1}=\widetilde{R}\cong\mathbb{P}^1\times \mathbb{P}^1$.
Moreover, we have $Z_0^2=0$ on the~surface $F_0$, and we have $Z_0^2=2$ on the~surface $F_{-1}$.
Then $F_1\cong\mathbb{F}_2$ by Lemma~\ref{lemma:iterated-blowup-Fn}.
Moreover, since $\mathrm{PGL}_2(\mathbb{C})$ acts faithfully on the~curve $Z_0$, it acts faithfully on the~surface $F_1$.
Furthermore, if $\mathscr{C}$ is a~$\mathrm{PGL}_2(\mathbb{C})$-invariant irreducible curve in $F_1$,
then it follows from \cite[Theorem~5.1]{Mabuchi} that either $\mathscr{C}=\overline{F}_0\cap F_1$ or $\mathscr{C}=\overline{F}_{-1}\cap F_1$.
Using Lemma~\ref{lemma:iterated-blowup-Fn} again, we see that
\begin{itemize}
\item if $\mathscr{C}=\overline{F}_0\cap F_1$, then $\mathscr{C}^2=2$ on the~surface $F_1$, and $\mathscr{C}^2=0$ on the~surface $\overline{F}_0$,
\item if $\mathscr{C}=\overline{F}_{-1}\cap F_1$, then $\mathscr{C}^2=-2$ on the~surface $F_1$, while $\mathscr{C}^2=2$ on the~surface~$\overline{F}_0$.
\end{itemize}
Thus, we conclude that our claim holds for $m=1$. This is the~base of induction.

Suppose that our claim  holds for $m\geqslant 1$. Let us show that it holds for $m+1$ blow ups.
Let $\mathscr{C}$ be a~$\mathrm{PGL}_2(\mathbb{C})$-invariant irreducible curve in $F_m$,
let $\Theta\colon \mathcal{V}\to V_m$ be its blow up,
and let $\mathcal{F}$ be the~$\Theta$-exceptional surface.
By induction, we know that $F_m\cong\mathbb{F}_n$ for $n>0$.
Moreover, we also know that
$$
\mathscr{C}=F_m\cap\overline{F}_r
$$
for some $r\in\{-1,0,1,\ldots,m-1\}$. Furthermore, one of the~following two assertions holds:
\begin{itemize}
\item either $\mathscr{C}^2=n>0$ on the~surface $F_m$, and $\mathscr{C}^2\leqslant 0$ on  the~surface $\overline{F}_r$,
\item or $\mathscr{C}^2=-n<0$ on  the~surface $F_m$, and $\mathscr{C}^2>0$ on  the~surface $\overline{F}_r$.
\end{itemize}

Let $\mathcal{F}_m$ and $\mathcal{F}_r$ be the~strict transforms on $\mathcal{V}$ of the~surfaces $F_m$ and $\overline{F}_r$, respectively.
Then $\mathcal{F}\cap\mathcal{F}_m$ and $\mathcal{F}\cap\mathcal{F}_r$ are disjoint $\mathrm{PGL}_2(\mathbb{C})$-invariant irreducible curves
that are sections of the~ projection $\mathcal{F}\to\mathscr{C}$.
Let $\gamma$ be the~self-intersection $\mathscr{C}^2$ on the~surface $\overline{F}_r$.
Then it follows from Lemma~\ref{lemma:iterated-blowup-Fn} that $F_{m+1}\cong\mathbb{F}_{s}$ for
$$
s=n+|\gamma|>0.
$$
Thus, by \cite[Theorem~5.1]{Mabuchi}, the~curves $\mathcal{F}\cap\mathcal{F}_m$ and $\mathcal{F}\cap\mathcal{F}_r$
are the~only $\mathrm{PGL}_2(\mathbb{C})$-invariant irreducible curves in the~surface $\mathcal{F}$.
Let $\mathcal{C}_1=\mathcal{F}\cap\mathcal{F}_m$ and $\mathcal{C}_2=\mathcal{F}\cap\mathcal{F}_r$.

Suppose that $\mathscr{C}^2=n$ on the~surface $F_m$. In this case, we have $\gamma\leqslant 0$ and $s=n-\gamma>0$.
By Lemma~\ref{lemma:iterated-blowup-Fn}, we have $\mathcal{C}_1^2=n>0$ on the~surface $\mathcal{F}_m$, and $\mathcal{C}_1^2=-s$ on the~surface~$\mathcal{F}$.
Similarly, we see that $\mathcal{C}_2^2=\gamma\leqslant 0$ on the~surface $\mathcal{F}_r$, and $\mathcal{C}_2^2=s>0$ on the~surface $\mathcal{F}$.
Thus, we see that the~required claim holds for $m+1$ blow ups in this case.

Finally, we suppose that $\mathscr{C}^2=-n$ on the~surface $F_m$. Then $\gamma>0$ and $s=n+\gamma>0$.
By~Lemma~\ref{lemma:iterated-blowup-Fn}, we have $\mathcal{C}_1^2=-n<0$ on the~surface $\mathcal{F}_m$, and $\mathcal{C}_1^2=s$ on the~surface~$\mathcal{F}$.
Similarly, we have $\mathcal{C}_2^2=\gamma>0$ on the~surface $\mathcal{F}_r$, and $\mathcal{C}_2^2=-s<0$ on the~surface $\mathcal{F}$.
Therefore, we proved that the~required claim holds for $m+1$ blow up also in this case.
Hence, it holds for any number of blow ups (by induction).
\end{proof}

By Lemmas~\ref{lemma:3-17-beta-R}, \ref{lemma:3-17-beta-E}, \ref{lemma:3-17-R-prime}, we have $\beta(R)>0$, $\beta(E)>0$, $\beta(R^\prime)>0$, respectively.
Thus, to~prove that $X$ is K-polystable, it is enough to check that $\beta(F)>0$ in the~following cases:
\begin{enumerate}
\item when $F$ is the~$(a,b)$-divisor between $E$ and $\widetilde{R}$,

\item when $F$ is the~$(a,b)$-divisor between  $\widehat{E}$ and $R^\prime$.
\end{enumerate}
We start with the~first case.

\begin{proposition}
\label{proposition:3-17-beta-E-R}
Let $\nu\colon Y\to\widetilde{X}$ be the~$(a,b)$-blow up between the~surfaces $E$ and~$\widetilde{R}$,
and let $F$ be the~$\nu$-exceptional surface. Then $\beta(F)>0$.
\end{proposition}

\begin{proof}
Let $\overline{E}_1$, $\overline{E}_2$, $\overline{E}$, $\overline{R}$ be the~proper transforms on $Y$ of the~surfaces $E_1$, $E_2$, $E$, $\widetilde{R}$, respectively.
Take a~non-negative real number $x$. Put $\eta=f\circ \nu$. Then
$$
\eta^*(-K_X)-xF\sim_{\mathbb{R}}\overline{E}_1+\overline{E}_2+\overline{R}+3\overline{E}+(a+3b-x)F,
$$
so that the~pseudoeffective threshold $\tau=\tau(F)$ is at least $a+3b$.

Suppose that $x<\tau$.
Then $\overline{E}$ lies in the~stable base locus of the~divisor $\eta^*(-K_X)-xF$.
Moreover, we claim that the~positive part of the~Zariski decomposition of this divisor  has the~following form:
$$
\eta^*(-K_X)-\frac{t}{a+b}\overline{E}-xF-D
$$
for an~effective $\mathbb{R}$-divisor $D$.
Indeed, if $\ell$ is a~general fiber of the~projection $\overline{E}\to C$, then
$$
\Big(\eta^*(-K_X)-xF\Big)\cdot\ell=-\frac{x}{a},
$$
because $\eta^*(-K_X)\cdot\ell=0$ and $F\cdot\ell=\frac{1}{a}$.
On the~other hand, we have $\overline{E}\cdot\ell=-\frac{a+b}{a}$, which implies the~required claim.
Thus, if $7b>2a$, then
$$
S_X(F)\leqslant (a+b)S_X(E)=\frac{11}{9}(a+b),
$$
because $S_X(E)=\frac{11}{9}$ by Lemma~\ref{lemma:3-17-beta-E}. Thus, if $\frac{b}{a}>\frac{2}{7}$, then
$$
\beta(F)=A_X(F)-S_X(F)=a+2b-S_X(F)\geqslant a+2b-\frac{11}{9}(a+b)=\frac{7b-2a}{9}>0
$$
as required. Hence, we may assume that $\frac{b}{a}\leqslant\frac{2}{7}$.

If $x>2b$, then the~surface $\overline{R}$ lies in the~stable base locus of the~divisor $\eta^*(-K_X)-xF$.
Moreover, in this case,  the~Zariski decomposition of this divisor has the~following the~form:
$$
\eta^*(-K_X)-\frac{t}{a+b}\overline{E}-\frac{t-2b}{a+b}\overline{R}-xF-D
$$
for some effective $\mathbb{R}$-divisor $D$ (supported in $\overline{E}_1$, $\overline{E}_2$, $\overline{E}$, $\overline{R}$, $F$).
Indeed, if $\ell$ is a~general fiber of the~natural projection $\overline{R}\to \phi(C)$. Then $\overline{R}\cdot\ell=-\frac{a+b}{b}$ and
$$
\Big(\eta^*(-K_X)-xF\Big)\cdot\ell=2-\frac{x}{b},
$$
which implies that the~Zariski decomposition has the~required form for $x>2b$. Then
\begin{multline*}
S_X(F)\leqslant\frac{1}{36}\int_0^{2b}\!\mathrm{vol}\left(\varphi^*(-K_X)-\frac{t}{a+b} E\right) dt\,+
 \frac{1}{36}\int_{2b}^\infty\!\mathrm{vol}\left(\varphi^*(-K_X)- \frac{t-2b}{a+b} R\right) dt=\\
=(a+b)\cdot S\left(E,\frac{2b}{a+b}\right)+(a+b)\cdot S(R)<\frac{5}{9}(a+b) + \frac{4}{9}(a+b)=a+b.
\end{multline*}
because we have $S(R)=\frac{4}{9}$ by Lemma~\ref{lemma:3-17-beta-R}, and we have $S\!\left(E,\frac{2b}{a+b}\right)<\frac{5}{9}$ by Lemma~\ref{lemma:3-17-beta-E}.
This gives $\beta(F)>0$, since $A_X(F)=a+2b$.
\end{proof}

Finally, we deal with $(a,b)$-divisors between  $\widehat{E}$ and $R^\prime$.

\begin{proposition}
\label{proposition:3-17-beta-E-R-prime}
Let $\nu\colon Y\to\widehat{X}$ be the~$(a,b)$-blow up between the~surfaces $\widehat{E}$ and~$R^\prime$,
and let $F$ be the~$\nu$-exceptional surface. Then $\beta(F)>0$.
\end{proposition}

\begin{proof}
Let $\overline{E}_1$, $\overline{E}_2$, $\overline{E}$, $\overline{R}$, $\overline{R}^\prime$ be the~proper transforms on $Y$ of $E_1$, $E_2$, $E$, $\widetilde{R}$,~$R^\prime$, respectively.
Take a~non-negative real number $x$. Put $\eta=f\circ g\circ \nu$. Then
$$
\eta^*(-K_X)-xF\sim_{\mathbb{R}}\overline{E}_1+\overline{E}_2+\overline{R}+3\overline{E}+3\overline{R}^\prime+(3a+3b-x)F,
$$
so that the~pseudoeffective threshold $\tau=\tau(F)$ is at least $3a+3b$.

Suppose that $x<\tau$. Then $\overline{E}$ lies in the~stable base locus of the~divisor $\eta^*(-K_X)-xF$.
Moreover, we claim that the~positive part of the~Zariski decomposition of this divisor  has the~following form:
$$
\eta^*(-K_X)-\frac{t}{2a+b}\overline{E}-xF-D
$$
for an~effective $\mathbb{R}$-divisor $D$.
Indeed, if $\ell$ is a~general fiber of the~projection $\overline{E}\to C$, then
$$
\Big(\eta^*(-K_X)-xF\Big)\cdot\ell=-\frac{x}{a},
$$
because $\eta^*(-K_X)\cdot\ell=0$ and $F\cdot\ell=\frac{1}{a}$.
On the~other hand, we have $\overline{E}\cdot\ell=-\frac{2a+b}{a}$, which implies the~required claim.
Thus, we have
$$
S_X(F)\leqslant (2a+b)S_X(E)=\frac{11}{9}(2a+b),
$$
because $S_X(E)=\frac{11}{9}$ by Lemma~\ref{lemma:3-17-beta-E}. Then
$$
\beta(F)=A_X(F)-S_X(F)=3a+2b-S_X(F)\geqslant 3a+2b-\frac{11}{9}(a+b)=\frac{5a+7b}{9}>0
$$
as required.
\end{proof}

Thus, we see that $\beta(F)>0$ for every $G$-invariant prime divisor $F$ over the~threefold~$X$.
Then $X$ is K-polystable by Theorem~\ref{theorem:Fujita-Li}.

\section{Blow up of a~complete intersection of two quadrics in a~conic}
\label{section:2-16}

Let $Q_1=\{f=0\}\subset\mathbb{P}^5$, where
$f=x_0x_3+x_1x_4+x_2x_5$,
and let $Q_2=\{g=0\}\subset\mathbb{P}^5$, where
$g=x_0^2+\omega x_1^2+\omega^2x_2^2+(x_3^2+\omega x_4^2+\omega^2x_5^2)+(x_0x_3+\omega x_1x_4+\omega^2x_2x_5)$,
and $\omega$ is a~primitive cubic root of unity. Let $V_4=Q_1\cap Q_2$. Then $V_4$ is smooth.
Let $G$ be a~subgroup in $\mathrm{Aut}(\mathbb{P}^5)$ such that $G\cong\mumu_2^2\rtimes\mumu_3$,
where the~generator of $\mumu_3$ acts by
$$
\big[x_0:x_1:x_2:x_3:x_4:x_5\big]\mapsto \big[x_1:x_2:x_0:x_4:x_5:x_3\big],
$$
the generator of the~first factor of $\mumu_2^2$ acts by
$$
\big[x_0:x_1:x_2:x_3:x_4:x_5\big]\mapsto\big[-x_0:x_1:-x_2:-x_3:x_4:-x_5\big],
$$
and the~generator of the~second factor of $\mumu_2^2$ acts by
$$
\big[x_0:x_1:x_2:x_3:x_4:x_5\big]\mapsto\big[-x_0:-x_1:x_2:-x_3:-x_4:x_5\big].
$$
Then $G\cong\mathfrak{A}_4$, and $\mathbb{P}^5=\mathbb{P}(\mathbb{U}_3\oplus\mathbb{U}_3)$,
where $\mathbb{U}_3$ is the~unique (unimodular) irreducible three-dimensional representation of the~group $G$.
Note that $Q_1$ and $Q_2$ are $G$-invariant, so that $V_4$ is also $G$-invariant.
Thus, we may identify $G$ with a~subgroup in $\mathrm{Aut}(V_4)$.

Let $\tau$ be an~involution in $\mathrm{Aut}(\mathbb{P}^5)$ that is given by
$$
\big[x_0:x_1:x_2:x_3:x_4:x_5\big]\mapsto \big[x_3:x_4:x_5:x_0:x_1:x_2\big].
$$
Then $Q_1$ and $Q_2$ are $\tau$-invariant, so that $V_4$ is also $\tau$-invariant.

The group $G$ does not have fixed points in $\mathbb{P}^5$, and there are no $G$-invariant lines in $\mathbb{P}^5$.
Moreover, every $G$-invariant plane in $\mathbb{P}^5$ is given by
$$
\left\{\aligned
&\lambda x_0+\mu x_3=0,\\
&\lambda x_1+\mu x_4=0,\\
&\lambda x_2+\mu x_5=0,\\
\endaligned
\right.
$$
where $[\lambda:\mu]\in\mathbb{P}^1$.
Using this, we see that $V_4$ contains exactly four $G$-invariant conics.
These conics are cut out on $V_4$ by the~following $G$-invariant planes:
the plane $\Pi_1$ given by $x_0=x_1=x_2=0$, the~plane $\Pi_2=\tau(\Pi_1)$,
the plane $\Pi_3$ given by
$$
\left\{\aligned
&x_0=\omega x_3,\\
&x_1=\omega x_4,\\
&x_2=\omega x_5,\\
\endaligned
\right.
$$
and the~plane $\Pi_4=\tau(\Pi_3)$.
We let $C_1=V_4\cap\Pi_1$, $C_2=V_4\cap\Pi_2$, $C_3=V_4\cap\Pi_3$, $C_4=V_4\cap\Pi_4$.
Then the~conics $C_1$, $C_2$, $C_3$, $C_4$ are pairwise disjoint, $C_2=\tau(C_1)$ and $C_4=\tau(C_3)$,

For every $i\in\{1,2,3,4\}$, we let $\pi_i\colon X_i\to V_4$ be the~blow up of the~conic $C_i$,
and we denote by $E_i$ the~exceptional surface of the~blow up $\pi_i$.
Then $X_1\cong X_2$ and $X_3\cong X_4$ are smooth Fano threefolds \textnumero 2.16,
and the~action of the~group $G$ lifts to its action on them.

For every $i\in\{1,2,3,4\}$, we have the~following $G$-equivariant diagram:
$$
\xymatrix{
&X_i\ar@{->}[ld]_{\pi_i}\ar@{->}[rd]^{\eta_i}&\\%
V_4\ar@{-->}[rr]&&\mathbb{P}^2}
$$
where the~dashed arrow is a~linear projection from the~plane $\Pi_i$,
and $\eta_i$ is a~conic bundle that is given by the~linear system $|\pi_i^*(H)-E_i|$,
where $H$ is a~hyperplane section of the~threefold $V_4$.
In each case, we have $\mathbb{P}^2=\mathbb{P}(\mathbb{U}_3)$.

\begin{lemma}
\label{lemma:E1-E2-E3-E4}
One has $E_1\cong E_2\cong E_3\cong E_4\cong\mathbb{P}^1\times\mathbb{P}^1$.
\end{lemma}

\begin{proof}
One has $E_i\cong\mathbb{F}_n$ for some integer $n\geqslant 0$. We have
$-E_i\vert_{E_i}\sim s_{E_i}+af_{E_i}$
where $s_{E_i}$ is a~section of the~projection $E_i\to C_i$ such that $s_{E_i}^2=-n$,
and $f_{E_i}$ is a~fiber of this projection. Since $E_i^3=2+K_{V_4}\cdot C_i=-2$, we have
$-2=(s_{E_i}+af_{E_i})^2=-n+2a$,
so that $a=\frac{n-2}{2}$. On the~other hand, we have $(\pi_i^*(H)-E_i)\vert_{E_i}\sim s_{E_i}+\frac{n+2}{2}f_{E_i}$.
Since $|\pi_i^*(H)-E_i|$ is base point free, we have $\frac{n+2}{2}\geqslant n$, so that either $n=0$ or $n=2$.
If~$n=2$, then $s_{E_i}$ is contracted by $\eta_i$ to a~point,
which is impossible, since $G$ does not have fixed points in $\mathbb{P}^2$.
Hence, we see that $n=0$, so that  $E_i\cong\mathbb{P}^1\times\mathbb{P}^1$.
\end{proof}

For each $i\in\{1,2,3,4\}$, let $\Delta_i$ be the~discriminant curve in $\mathbb{P}^2$ of the~conic bundle $\eta_i$.
Then $\Delta_i$ is a~(possibly reducible) quartic curve with at most ordinary double points.

\begin{lemma}
\label{lemma:2-16-discriminant}
The curves $\Delta_1$, $\Delta_2$, $\Delta_3$, $\Delta_4$ are smooth.
\end{lemma}

\begin{proof}
If $i=1$, then the~linear projection $V_4\dasharrow\mathbb{P}^2$ from the~plane $\Pi_1$ is given by
$$
\big[x_0:x_1:x_2:x_3:x_4:x_5\big]\mapsto \big[x_0:x_1:x_2\big].
$$
Using this, one can deduce that $\Delta_1$ is given by
$4x_0^4-x_0^2x_1^2-x_0^2x_2^2+4x_1^4-x_1^2x_2^2+4x_2^4=0$.
This curve is smooth. Thus, the~curve $\Delta_2\cong\Delta_1$ is also smooth.

Let $y_0=x_0-\omega x_3$, $y_1=x_1-\omega x_4$, $y_2=x_2-\omega x_5$, $y_3=x_3$, $y_4=x_4$, $y_5=x_5$.
In new coordinates, the~linear projection $V_4\dasharrow\mathbb{P}^2$ from the~plane $\Pi_3$ is given by
$$
\big[y_0:y_1:y_2:y_3:y_4:y_5\big]\mapsto \big[y_0:y_1:y_2\big].
$$
Then $\Delta_3$ is given by
$4x_0^4-\omega x_0^2x_1^2+(\omega+1)x_2^2x_0^2-4(\omega+1)x_1^4-x_1^2x_2^2+4\omega x_2^4$.
This curve is smooth, so that $\Delta_4\cong\Delta_3$ is also smooth.
\end{proof}

Observe that $\mathbb{P}^2=\mathbb{P}(\mathbb{U}_3)$ has three $G$-invariant conics.
Denote them by $\mathcal{C}_1$, $\mathcal{C}_2$ and $\mathcal{C}_3$,
and denote by $F_{1,i}$, $F_{2,i}$ and $F_{3,i}$ their preimages on $X_i$ via $\eta_i$, respectively.
Then
$$
F_{1,i}\sim F_{2,i}\sim F_{3,i}\sim \pi_i^*(2H)-2E_i.
$$
For every $i\in\{1,2,3,4\}$ and $j\in\{1,2,3\}$, let $\overline{F}_{j,i}=\pi_i(F_{j,i}))$.
Then $\overline{F}_{j,i}$ is an~irreducible surface in $|2H|$ that is singular along the~conic $C_i$.
Without loss of generality, we may assume that  $\overline{F}_{1,1}$ is cut out on $V_4$ by the~equation $f_{1,1}=0$ for
$f_{1,1}=x_0^2+x_1^2+x_3^2$,
and the~surface $\overline{F}_{2,1}$ is cut out on $V_4$ by the~equation $f_{2,1}=0$ for
$f_{2,1}=x_0^2+\omega x_1^2+\omega^2 x_3^2$.
Then the~surface $\overline{F}_{3,1}$ is cut out on $V_4$ by the~equation $f_{3,1}=0$, where $f_{3,1}=x_0^2+\omega^2 x_1^2+\omega x_3^2$.
Using the~involution $\tau$, we also see that $\overline{F}_{1,2}=\tau(\overline{F}_{1,1})$, $\overline{F}_{2,2}=\tau(\overline{F}_{2,1})$ and $\overline{F}_{3,2}=\tau(\overline{F}_{3,1})$,
so that we let $f_{1,2}=\tau^*(f_{1,1})$, $f_{2,2}=\tau^*(f_{2,1})$ and $f_{3,2}=\tau^*(f_{3,1})$.
Then $\overline{F}_{1,3}$ is cut out by $f_{1,3}=0$, where
$f_{1,3}=(x_0-\omega x_3)^2+(x_1-\omega x_4)^2+(x_2-\omega x_5)^2$.
Likewise, the~surface $\overline{F}_{2,3}$ is cut out on $V_4$ by the~equation $f_{2,3}=0$, where
$f_{2,3}=(x_0-\omega x_3)^2+\omega(x_1-\omega x_4)^2+\omega^2(x_2-\omega x_5)^2$,
Similarly, $\overline{F}_{3,3}$ is cut out by $f_{3,3}=0$, where
$f_{3,3}=(x_0-\omega x_3)^2+\omega^2(x_1-\omega x_4)^2+\omega(x_2-\omega x_5)^2$.
Finally, we conclude that $\overline{F}_{1,4}=\tau(\overline{F}_{1,3})$, $\overline{F}_{2,4}=\tau(\overline{F}_{2,3})$ and $\overline{F}_{3,4}=\tau(\overline{F}_{3,3})$,
so that we let $f_{1,4}=\tau^*(f_{1,3})$, $f_{2,4}=\tau^*(f_{2,3})$ and $f_{3,4}=\tau^*(f_{3,3})$.

\begin{remark}
\label{remark:2-16-surfaces-curves}
The incidence relation between the~surfaces
$\overline{F}_{1,1}$, $\overline{F}_{2,1}$, $\overline{F}_{3,1}$,
$\overline{F}_{1,2}$, $\overline{F}_{2,2}$, $\overline{F}_{3,2}$,
$\overline{F}_{1,3}$, $\overline{F}_{2,3}$, $\overline{F}_{3,3}$,
$\overline{F}_{1,4}$, $\overline{F}_{2,4}$, $\overline{F}_{3,4}$ and the~conics $C_1$, $C_2$, $C_3$, $C_4$
is described in the~following table:
\begin{center}\renewcommand{\arraystretch}{1.5}
\begin{tabular}{|c||c|c|c|c|c|c|c|c|c|c|c|c|}
  \hline
&$\overline{F}_{1,1}$&$\overline{F}_{2,1}$&$\overline{F}_{3,1}$&$\overline{F}_{1,2}$&$\overline{F}_{2,2}$&$\overline{F}_{3,2}$&$\overline{F}_{1,3}$&$\overline{F}_{2,3}$&$\overline{F}_{3,3}$&$\overline{F}_{1,4}$&$\overline{F}_{2,4}$&$\overline{F}_{3,4}$ \\
  \hline
  \hline
$C_1$ & $\mathrm{Node}$  & $\mathrm{Node}$   & $\mathrm{Cusp}$  & $\mathrm{No}$  & $\mathrm{Yes}$  & $\mathrm{No}$  & $\mathrm{No}$  & $\mathrm{Yes}$  & $\mathrm{No}$  & $\mathrm{No}$  & $\mathrm{Yes}$  & $\mathrm{No}$ \\
  \hline
$C_2$ & $\mathrm{No}$  & $\mathrm{Yes}$  & $\mathrm{No}$  & $\mathrm{Node}$   & $\mathrm{Node}$   & $\mathrm{Cusp}$  & $\mathrm{No}$  & $\mathrm{Yes}$  & $\mathrm{No}$  & $\mathrm{No}$  & $\mathrm{Yes}$  & $\mathrm{No}$   \\
  \hline
$C_3$ & $\mathrm{Yes}$  & $\mathrm{No}$  & $\mathrm{No}$  & $\mathrm{Yes}$  & $\mathrm{No}$  & $\mathrm{No}$  & $\mathrm{Node}$    & $\mathrm{Node}$   & $\mathrm{Cusp}$  & $\mathrm{Yes}$  & $\mathrm{No}$  & $\mathrm{No}$  \\
  \hline
$C_4$ & $\mathrm{Yes}$  & $\mathrm{No}$  & $\mathrm{No}$  & $\mathrm{Yes}$  & $\mathrm{No}$  & $\mathrm{No}$  & $\mathrm{Yes}$  & $\mathrm{No}$  & $\mathrm{No}$  & $\mathrm{Node}$    & $\mathrm{Node}$   & $\mathrm{Cusp}$ \\
  \hline
\end{tabular}
\end{center}
\medskip
\noindent
Here, $\mathrm{No}$ means that the~surface does not contains the~conic,
and in all other cases the~surface contains the~conic.
Likewise, $\mathrm{Node}$ means the~the surface has an~ordinary double point in general point of the~conic,
and $\mathrm{Cusp}$ means that  the~surface has an~ordinary cusp in general point of the~conic.
In all remaining cases the~surface is smooth at general point of the~conic (we will see later that it is smooth along this conic).
\end{remark}

\begin{corollary}
\label{corollary:2-16-upper-bound}
For every $i\in\{1,2,3,4\}$, one has $\alpha_G(X_i)\leqslant\frac{3}{4}$.
\end{corollary}

\begin{proof}
Observe that $F_{3,i}+E_i\sim -K_{X_i}$.
Moreover, it follows from Remark~\ref{remark:2-16-surfaces-curves} that
the surface $F_{3,i}$ is tangent to $E_i$ along a~section of the~projection $E_i\to C_i$.
Thus, we conclude that $\alpha_G(X_i)\leqslant\mathrm{lct}(X_i,F_{3,i}+E_i)\leqslant\frac{3}{4}$
as required.
\end{proof}

Recall that the~group $G\cong\mumu_2^2\rtimes\mumu_3$ has three different one-dimensional representations:
the trivial representation with the~character $\chi_0$,
the non-trivial representation with the~character $\chi_1$ that sends
the generator of $\mumu_3$ to $\omega$, and the~non-trivial representation with
the character $\chi_2$ that sends the~generator of $\mumu_3$ to $\omega^2$.
On the~other hand, the~polynomials $f$, $g$, $f_{1,1}$, $f_{2,1}$, $f_{3,1}$,
$f_{1,2}$, $f_{2,2}$, $f_{3,2}$, $f_{1,3}$, $f_{2,3}$, $f_{3,3}$, $f_{1,4}$, $f_{2,4}$, $f_{3,4}$
are semi-invariants of the~group $G$ considered as a~subgroup in $\mathrm{SL}_6(\mathbb{C})$.
These polynomials split into three groups with respect to the~characters $\chi_0$, $\chi_1$ and $\chi_2$ as follows:
\begin{enumerate}
\item[($\chi_0$)] $f$, $f_{1,1}$, $f_{1,2}$, $f_{1,3}$, $f_{1,4}$ are $G$-invariants;
\item[($\chi_1$)] $f_{3,1}$, $f_{3,2}$, $f_{3,3}$, $f_{3,4}$ are $G$-semi-invariants with character $\chi_1$;
\item[($\chi_2$)] $g$, $f_{2,1}$, $f_{2,2}$, $f_{2,3}$, $f_{2,4}$ are $G$-semi-invariants with character $\chi_2$.
\end{enumerate}
Note that  $f_{1,4}=-(\omega+2)f_{1,1}+(\omega+2)f_{1,2}+f_{1,3}$ and
$(\omega+1)f_{1,1}-\omega f_{1,2}-(\omega+1)f_{1,3}+2f=0$,
which implies that $\overline{F}_{1,1}$, $\overline{F}_{1,2}$, $\overline{F}_{1,3}$, $\overline{F}_{1,4}$ generate a~pencil on $V_4$,
which we denote by $\mathcal{P}_0$.
Similarly, we have
$f_{3,4}=-(\omega+2)f_{3,1}+(\omega+2)f_{3,2}+f_{3,3}$,
and the~surfaces $\overline{F}_{3,1}$, $\overline{F}_{3,2}$, $\overline{F}_{3,3}$, $\overline{F}_{3,4}$
generate two-dimensional linear system (net), which we denote by $\mathcal{M}_1$.
This linear system $\mathcal{M}_1$ contains four pencils, which we denote by $\mathcal{P}_{1,1}$, $\mathcal{P}_{1,2}$, $\mathcal{P}_{1,3}$ and $\mathcal{P}_{1,4}$,
that consist of surfaces containing the~conics $C_1$, $C_2$, $C_3$ and $C_4$, respectively.
Likewise, we have $f_{2,4}=-(\omega+2)f_{2,1}+(\omega+2)f_{2,2}+f_{2,3}$
and
$(\omega-1)f_{2,1}-(\omega+2)f_{2,2}-(\omega+1)f_{2,3}+2g=0$,
so that $\overline{F}_{2,1}$, $\overline{F}_{2,2}$, $\overline{F}_{2,3}$, $\overline{F}_{2,4}$ generates a~pencil on $V_4$,
which we denote by $\mathcal{P}_2$.

For every $i\in\{1,2,3,4\}$, denote by $\mathcal{P}_0^i$,
$\mathcal{P}_{1,1}^i$, $\mathcal{P}_{1,2}^i$, $\mathcal{P}_{1,3}^i$, $\mathcal{P}_{1,4}^i$ and  $\mathcal{P}_2^i$
the strict transforms on $X_i$ of the~pencils $\mathcal{P}_0$,
$\mathcal{P}_{1,1}$, $\mathcal{P}_{1,2}$, $\mathcal{P}_{1,3}$, $\mathcal{P}_{1,4}$ and  $\mathcal{P}_2$.
Then
\begin{align*}
\mathcal{P}_{1,1}^1&\sim \mathcal{P}_{2}^1\sim -K_{X_1},\\
\mathcal{P}_{1,2}^2&\sim \mathcal{P}_{2}^2\sim -K_{X_2},\\
\mathcal{P}_{1,3}^3&\sim \mathcal{P}_{0}^3\sim -K_{X_3}, \\
\mathcal{P}_{1,4}^4&\sim \mathcal{P}_{0}^4\sim -K_{X_4}.
\end{align*}
Moreover, we have $F_{3,1}+E_1\in\mathcal{P}_{1,1}^1$, $F_{2,1}+E_1\in\mathcal{P}_{2}^1$ $F_{3,2}+E_2\in\mathcal{P}_{1,2}^2$, $F_{2,2}+E_2\in\mathcal{P}_{2}^2$,
$F_{3,3}+E_3\in\mathcal{P}_{1,3}^3$, $F_{1,3}+E_3\in\mathcal{P}_{0}^3$, $F_{3,4}+E_4\in\mathcal{P}_{1,4}^4$, $F_{1,4}+E_3\in\mathcal{P}_{0}^4$.
Thus, we see that the~restrictions
$\mathcal{P}_{1,1}^1\vert_{X_1}$, $\mathcal{P}_{2}^1\vert_{X_1}$, $\mathcal{P}_{1,2}^2\vert_{X_2}$,
$\mathcal{P}_{2}^2\vert_{X_2}$, $\mathcal{P}_{1,3}^3\vert_{X_3}$, $\mathcal{P}_{0}^3\vert_{X_3}$,
$\mathcal{P}_{1,4}^4\vert_{X_4}$, $\mathcal{P}_{0}^4\vert_{X_4}$
are $G$-invariant curves in $E_1$, $E_2$, $E_3$, $E_4$, respectively.
Denote them by $Z_1$, $Z_1^\prime$,  $Z_2$, $Z_2^\prime$,  $Z_3$, $Z_3^\prime$,  $Z_4$, $Z_4^\prime$, respectively.
Observe that $Z_1\ne Z_1^\prime$,  $Z_2\ne Z_2^\prime$,  $Z_3\ne Z_3^\prime$ and $Z_4\ne Z_4^\prime$.
This follows from the~exact sequence of $G$-representations
$$
0\rightarrow H^0\Big(\mathcal{O}_{X_i}\big(-K_{X_i}-E_i\big)\Big)\rightarrow H^0\Big(\mathcal{O}_{X_i}\big(-K_{X_i}\big)\Big)
\twoheadrightarrow H^0\Big(\mathcal{O}_{E_i}\big(-K_{X_i}\big\vert_{E_i}\big)\Big),
$$
where the~surjectivity of the~last map follows from Kodaira vanishing.
Alternatively, one can show this using the~explicit equations of the~pencils $\mathcal{P}_0$,
$\mathcal{P}_{1,1}$, $\mathcal{P}_{1,2}$, $\mathcal{P}_{1,3}$, $\mathcal{P}_{1,4}$ and  $\mathcal{P}_2$.

Recall that $E_1\cong E_2\cong E_3\cong E_4\cong\mathbb{P}^1\times\mathbb{P}^1$ by Lemma~\ref{lemma:E1-E2-E3-E4}.
For every $i\in\{1,2,3,4\}$, let $s_{E_i}$ be a~section of the~projection $E_i\to C_i$ such that $s_{E_i}^2=0$,
and let $f_{E_i}$ be  a~fiber of this projection.
Then $-E_i\vert_{E_i}=s_{E_i}-f_{E_i}$,
so that $-K_{X_i}\sim s_{E_i}+3f_{E_i}$. Hence, we see that $Z_i\sim Z_i^\prime \sim s_{E_i}+3f_{E_i}$,
which immediately implies that both curve $Z_i$ and $Z_i^\prime$ are irreducible,
because $C_i$ does not have $G$-orbits of lengths $1$, $2$ and $3$.

For each $i\in\{1,2,3,4\}$, the~conic bundle $\eta_i$ gives a~double cover $E_i\to\mathbb{P}^2$,
whose branching curve is $\mathcal{C}_3$.
Indeed, one has $F_{3,i}\sim\pi_i^*(2H)-2E_i$,
and $\overline{F}_{3,i}$ has a~cusp at general point of the~conic $C_i$.
Since $F_{3,i}\vert_{E_i}\sim 2s_{E_i}+2f_{E_i}$, we have $\overline{F}_{3,i}\vert_{E_i}=2C_i^i$
for some irreducible curve $C_i^i\in|s_{E_i}+f_{E_i}|$.
Since the~double cover $E_i\to\mathbb{P}^2$ is given by a~linear subsystem in $|s_{E_i}+f_{E_i}|$,
we conclude that $\eta_i(C_i^i)$ is the~branching curve of this double cover.
But $\eta_i(C_i^i)=\mathcal{C}_3$, since $F_{3,i}$ is the~preimage of the~curve $\mathcal{C}_3$ via $\eta_i$.

For every $i$ and $j$ in $\{1,2,3,4\}$ such that $j\ne i$, denote by
$C_j^i$ the~strict transform of the~conic $C_j$ on the~threefold $X_i$. Then
$-K_{X_i}\cdot C_1^i=-K_{X_i}\cdot C_2^i=-K_{X_i}\cdot C_3^i=-K_{X_i}\cdot C_4^i=4$
and $-K_{X_i}\cdot Z_i=-K_{X_i}\cdot Z_i^\prime=6$.
Observe also that $C_1^i$, $C_2^i$, $C_3^i$, $C_4^i$, $Z_i$, $Z_i^\prime$ are smooth rational curves.
Moreover, we have the~following result:

\begin{lemma}
\label{lemma:curves-small-degree}
Let $C$ be an~irreducible $G$-invariant curve in $X_i$ such that $C\cong\mathbb{P}^1$ and $-K_{X_i}\cdot C<8$.
Then $C$ is one of the~curves $C_1^i$, $C_2^i$, $C_3^i$, $C_4^i$, $Z_i$, $Z_i^\prime$.
\end{lemma}

\begin{proof}
The proof is  the~same for every $i\in\{1,2,3,4\}$.
Thus, for simplicity of notations, we assume that $i=1$.
Suppose that $C$ is not one of the~curves $C_1^1$, $C_2^1$, $C_3^1$, $C_4^1$, $Z_1$, $Z_1^\prime$.
Let us seek for a~contradiction.

First, we suppose that $C\subset E_1$. Then $C\sim as_{E_1}+bf_{E_1}$ for some non-negative integers $a$ and $b$.
Since $-K_{X_1}\vert_{E_1}\sim s_{E_1}+3f_{E_1}$, we see that $3a+b=-K_{X_i}\cdot C<8$.
Moreover, since $C_1^1\cdot C=a+b$, we conclude that $a+b\geqslant 4$ and $a+b\ne 5$,
because $C_1^1$ does not have $G$-orbits of lengths $1$, $2$, $3$ and $5$.
Thus, since $C$ is irreducible, we conclude that $a=1$ and $b=3$.

Let us describe the~action of $G$ on the~surface $E_1\cong\mathbb{P}^1\times\mathbb{P}^1$.
Since $G$ acts faithfully on $C_1\cong\mathbb{P}^1$, this action is given
by the~unique (unimodular) irreducible two-dimensional representation of the~central extension
$2.G\cong\mathrm{SL}_2(\mathbb{F}_3)$ of the~group $G$, which we denote by $\mathbb{W}_3$.
Since $|s_{E_1}+f_{E_1}|$ contains a~$G$-invariant curve, and the~projection $E_1\to C_1$ is $G$-equivariant, and
we deduce that the~action of $G$ on the~surface $E_1$ is given by
the identification $E_1=\mathbb{P}(\mathbb{W}_2)\times\mathbb{P}(\mathbb{W}_2)$.
Thus, the~$G$-invariant curves in  $|s_{E_1}+3f_{E_1}|$
corresponds to one-dimensional subrepresentations of the~group $2.G$ in
$\mathbb{W}_2\otimes\mathrm{Sym}^3(\mathbb{W}_2)$.
Using the~following GAP script, we conclude that there are two such subrepresentations:
\begin{center}
\begin{verbatim}
G:=Group("SL(2,3)");
R:=IrreducibleModules(G,CyclotomicField(3));
M:=TensorProduct(R[4],SymmetricPower(R[4],3));
IndecomposableSummands(M);
\end{verbatim}
\end{center}
These subrepresentations corresponds to the~curves $Z_1$ and $Z_1^\prime$, so that $C$ must be one of them,
which is impossible by assumption.

Thus, we see that $C$ is not contained in $E_1$.
Let $\overline{C}=\pi_1(C)$. Then $\pi_1^*(H)\cdot C=H\cdot\overline{C}\geqslant 2$.
Moreover, if $H\cdot\overline{C}=2$, then $\overline{C}$ is one of the~conics $C_1$, $C_2$, $C_3$ or $C_4$,
because these are the~only $G$-invariant conics in $V_4$.
Since $C\not\subset E_1$ and $C$  is not one of the~curves $C_2^1$, $C_3^1$, $C_4^1$,
we see that $H\cdot\overline{C}\ne 2$, so that $\pi_1^*(H)\cdot C\geqslant 3$.

Note also that $\eta_1(C)$ is a~curve, because $G$ does not have fixed points in $\mathbb{P}^2$.
Similarly, we see that $\eta_1(C)$ is not a~line. Hence, we conclude that
$(\pi_1^*(H)-E_1)\cdot C\geqslant\mathrm{deg}(\eta_1(C))\geqslant 2$.
One the~other hand, we have $E_1\cdot C$ must be even since $C$ does not have $G$-orbits of odd length.
Moreover, we have
$$
7\geqslant-K_{X_1}\cdot C=\big(\pi_1^*(2H)-E_1\big)\cdot C=\pi_1^*(H)\cdot C+\big(\pi_1^*(H)-E_1\big)\cdot C\geqslant 5,
$$
so that $-K_{X_1}\cdot C=6$, $\pi_1^*(H)\cdot C=3$ and $(\pi_1^*(H)-E_1)\cdot C=3$, which gives $E_1\cdot C=0$.
Hence, we see that $\overline{C}$ is a~smooth rational cubic curve,
and $\eta_1(C)$ is a~singular cubic curve.
This is impossible, since $G$ does not have fixed points in $\mathbb{P}^2$.
\end{proof}

\begin{lemma}
\label{lemma:2-16-surfaces}
Let $S$ be a $G$-invariant surface such that
$-K_{X_{i}}\sim_{\mathbb{Q}}aS+\Delta$ for a~rational number $a$ and an~effective $G$-invariant $\mathbb{Q}$-divisor $\Delta$ on $X_i$.
Then $a\leqslant 1$.
\end{lemma}

\begin{proof}
If $S=E_i$, then
$2=-K_{X_{i}}\cdot \mathscr{C}=aS\cdot\mathscr{C}+\Delta\cdot\mathscr{C}\geqslant a E_i\cdot\mathscr{C}=2a$
for a~general fiber $\mathscr{C}$ of the~conic bundle $\nu_i$. Thus, we may assume that $S\ne E_i$.
Then $\pi_i(S)$ is a~surface in $V_4$, and $2H\sim_{\mathbb{Q}}a\pi_i(S)+\pi_i(\Delta)$.
Hence, if $a>1$, then $\pi_i(Z)\sim H$,
which is impossible, because $\mathbb{P}^5$ does not contain $G$-invariant hyperplanes.
\end{proof}

Now we are ready to state the~main technical result of this section:

\begin{lemma}
\label{lemma:2-16}
Let $a$ and $\lambda$ be positive rational numbers such that $a\geqslant 1$ and $\lambda<\frac{3}{4}$,
and let $D$ be an~effective $G$-invariant $\mathbb{Q}$-divisor on $X_i$ such that $D\sim_{\mathbb{Q}}\pi_i^*(2H)-aE_i$.
Then $E_i$, $C_i^i$, $Z_i$ and $Z_i^\prime$ are not log canonical centers of the~log pair $(X_i,\lambda D)$.
\end{lemma}

Let us use this result to prove

\begin{proposition}
\label{proposition:2-16}
One has $\alpha_G(X_1)=\alpha_G(X_2)=\alpha_G(X_3)=\alpha_G(X_4)=\frac{3}{4}$.
\end{proposition}

\begin{proof}
Suppose that $\alpha_G(X_i)<\frac{3}{4}$. Let us seek for a~contradiction.
Since $X_i$ does not have $G$-fixed points,
it follows from \cite[Lemma~A.4.8]{ACCFKMGSSV} and Lemma~\ref{lemma:2-16-surfaces}
that there exists a~$G$-invariant $\mathbb{Q}$-divisor $D$ on the~threefold $X_i$ such that
$D\sim_{\mathbb{Q}} -K_{X_i}$,
the log pair $(X_i,\lambda D)$ is strictly log canonical for some positive rational number $\lambda<\frac{3}{4}$,
and the~only center of log canonical singularities of this log pair is an~irreducible $G$-invariant
smooth irreducible rational curve $Z\subset X_i$ such that $-K_{X_i}\cdot Z<8$.
Then it must be one of the~curves $C_1^i$, $C_2^i$, $C_3^i$, $C_4^i$, $Z_i$, $Z_i^\prime$ by Lemma~\ref{lemma:curves-small-degree}.
On the~other hand, it follows from  Lemma~\ref{lemma:2-16} that $Z$ is not one of the~curves $C_i^i$, $Z_i$, $Z_i^\prime$,
so that $Z=C_j^i$ for some $j\in\{1,2,3,4\}$ such that $j\ne i$.

Let $\nu\colon V\to X_i$ be the~blow up of the~curve $Z$, let $F$ be the~$\nu$-exceptional surface,
let~$\widetilde{D}$~be strict transform of the~divisor $D$ via $\nu$, and let $m=\mathrm{mult}_{Z}(D)$.
Then $m\geqslant\frac{1}{\lambda}$ and
$$
K_{V}+\lambda\widetilde{D}+\big(\lambda m-1\big)F\sim_{\mathbb{Q}}\nu^*\big(K_{X_i}+\lambda D\big).
$$
Thus, either $\lambda m-1\geqslant 1$ or
the surface $F$ contains an~irreducible $G$-invariant smooth rational curve $\widetilde{Z}$~such that $\nu(\widetilde{Z})=Z$,
the curve $\widetilde{Z}$ is a~section of the~projection $F\to Z$,
and $\widetilde{Z}$ is a~center of log canonical singularities the~log pair $(V,\lambda\widetilde{D}+(\lambda m-1)F)$.

Let $\upsilon\colon V\to X_j$ be the~birational contraction of the~strict transform of the~surface $E_i$,
and let $\overline{D}=\upsilon(\widetilde{D})$.
Then $\upsilon(F)=E_j$ and $\overline{D}\sim_{\mathbb{Q}}\pi_j(2H)-mE_j$, so that
$$
\overline{D}+\Big(m-\frac{1}{\lambda}\Big)E_j\sim_{\mathbb{Q}}\pi_j(2H)-\frac{1}{\lambda}E_j.
$$
Then the~surface $E_j$ and the~curves $C_j^j$, $Z_j$ and $Z_j^\prime$ are not log canonical centers of the~log pair $(X_j,\lambda\overline{D}+(\lambda m-1)E_j)$ by Lemma~\ref{lemma:2-16}.
In particular, we see that $\lambda m-1<1$, so that
the surface $E_j$ contains an~irreducible $G$-invariant smooth rational curve $\overline{Z}$~such that $\pi_j(\overline{Z})=Z$,
the curve $\overline{Z}$ is a~section of the~projection $E_j\to C_j$,
and $\overline{Z}$ is a~center of log canonical singularities of the~log pair $(X_j,\lambda\overline{D}+(\lambda m-1)E_j)$.
Let us repeat that the~curve $\overline{Z}$ is not one of the~curves $C_j^j$, $Z_j$ and $Z_j^\prime$ by Lemma~\ref{lemma:2-16}.

Recall that $E_j\cong\mathbb{P}^1\times\mathbb{P}^1$. Write
$\overline{D}\vert_{E_j}=\delta Z+\Upsilon$.
where $\delta$ is a~non-negative rational number,
and $\Upsilon$ is an~effective $\mathbb{Q}$-divisor on $E_j$ such that its support does not contain the~curve $Z$.
Then $\delta\geqslant\frac{1}{\lambda}>\frac{4}{3}$ by \cite[Theorem~5.50]{KoMo98}. But
$$
\overline{D}\big\vert_{E_j}\sim_{\mathbb{Q}}\Big(\pi_j(2H)-mE_j\Big)\big\vert_{E_j}\sim_{\mathbb{Q}} 4f_{E_j}+m(s_{E_j}-f_{E_j})=ms_{E_j}+(4-m)f_{E_j},
$$
and $Z\sim s_{E_j}+kf_{E_j}$ for some non-negative integer $k$. This gives
$$
\Upsilon\sim_{\mathbb{Q}}ms_{E_j}+(4-m)f_{E_j}-\delta\big(s_{E_j}+kf_{E_j}\big)=(m-\delta)s_{E_j}+(4-m-\delta k)f_{E_j}.
$$
Since $m\geqslant\frac{1}{\lambda}>\frac{4}{3}$ and $\delta>\frac{4}{3}$, we get $k=0$ or $k=1$,
so that $Z=C_{j}^j$ by Lemma~\ref{lemma:curves-small-degree}, which is impossible by Lemma~\ref{lemma:2-16}.
\end{proof}

By Proposition~\ref{proposition:2-16} and  Theorem~\ref{theorem:Tian},
the~smooth Fano threefolds $X_1\cong X_2$ and $X_3\cong X_4$ are K-polystable.
However, to complete the~proof of Proposition~\ref{proposition:2-16}, we have to prove technical Lemma~\ref{lemma:2-16}.
Note that it is enough to prove this lemma for $X_1$ and $X_3$,
so that we will assume in the~following that either $i=1$ or $i=3$.

Fix rational numbers $a$ and $\lambda$ such that $a\geqslant 1$ and $0<\lambda<\frac{3}{4}$.
Let $D$ be a~$G$-invariant effective $\mathbb{Q}$-divisor on the~threefold $X_i$ such that $D\sim_{\mathbb{Q}}\pi_i^*(2H)-aE_i$.
Then we must show that $E_i$, $C_i^i$, $Z_i$ and $Z_i^\prime$ are also not log canonical centers of the~ pair $(X_i,\lambda D)$.
Replacing $D$ by $D+(a-1)E_i$, we may assume that $a=1$, so that $D\sim_{\mathbb{Q}} -K_{X_i}$.
Write
$D=\varepsilon E_i+\Delta$,
where $\varepsilon\in\mathbb{Q}_{\geqslant 0}$,
and $\Delta$ is effective $\mathbb{Q}$-divisor on $X_i$ whose support does not contain $E_i$.
Then $\varepsilon\leqslant 1$ by Lemma~\ref{lemma:2-16-surfaces},
so that $E_i$ is not a~log canonical center of the~log pair $(X_i,\lambda D)$.

\begin{lemma}
\label{lemma:2-16-second}
Neither $Z_i$ nor $Z_i^\prime$ is a~log canonical center of the~pair $(X_i,\lambda D)$.
\end{lemma}

\begin{proof}
Denote by $Z$ one of the~curves $Z_i$ or $Z_i^\prime$. Let $m_{\Delta}=\mathrm{mult}_{Z}(\Delta)$ and $m=\mathrm{mult}_{Z}(D)$.
Then $m=m_{\Delta}+\varepsilon$. Let us bound $m$. To do this, write $\Delta\big\vert_{E_i}=\delta Z+\Upsilon$,
where $\delta$ is a~rational number such that $\delta\geqslant m_{\Delta}$,
and $\Upsilon$ is an~effective $\mathbb{Q}$-divisor on the~surface $E_i\cong\mathbb{P}^1\times\mathbb{P}^1$ such that its support does not contain $Z$.
Observe that
$$
\Delta\big\vert_{E_i}\sim_{\mathbb{Q}}\Big(\pi_i(2H)-(1+\varepsilon)E_i\Big)\big\vert_{E_i}\sim_{\mathbb{Q}} 4f_{E_i}+(1+\varepsilon)(s_{E_i}-f_{E_i})=(1+\varepsilon)s_{E_i}+(3-\varepsilon)f_{E_i}
$$
and $Z\sim s_{E_i}+3f_{E_i}$. This gives
$\Upsilon\sim_{\mathbb{Q}}(1+\varepsilon-\delta)s_{E_i}+(3-\varepsilon-3\delta)f_{E_i}$,
which gives $\delta\leqslant 1-\frac{\varepsilon}{3}$. In particular,  we get $m=m_{\Delta}+\varepsilon\leqslant\delta+\varepsilon\leqslant 1+\frac{2\varepsilon}{3}\leqslant\frac{5}{3}$.

Let $\nu\colon V\to X_i$ be the~blow up of the~curve $Z$, and let $F$ be the~$\nu$-exceptional surface.
Then the~action of the~group $G$ lifts to the~threefold $V$, since $Z$ is $G$-invariant.

Recall that $Z$ is cut out on $E_i$ by a~$G$-invariant surface in $|-K_{X_i}|$.
Since $Z\cong\mathbb{P}^1$, this gives
$\mathcal{N}_{Z/X_i}\cong\mathcal{O}_{\mathbb{P}^1}(6)\oplus\mathcal{O}_{\mathbb{P}^1}(-2)$,
because $-K_{X_{i}}\cdot  Z=6$, and $Z^2=6$ on the~surface $E_i$. Thus, we have $F\cong\mathbb{F}_8$.
Moreover, since $F^3=-4$, we deduce that
$-F\big\vert_F\sim s_F+2f_F$,
where $s_F$ is a~section of the~projection $F\to Z$ such that $s_F^2=-8$, and $f_F$ is a~fiber of this projection.
Let $\widetilde{E}_i$ and $\widetilde{D}$ be the~proper transforms of the~divisors $E_i$ and $D$ on the~threefold $V$, respectively.
Then $\widetilde{E}_i\vert_F\sim(\nu^*(E_i)-F)\vert_F\sim s_F$,
since $E\cdot Z=-2$. Thus, we see that $\widetilde{E}_i\vert_F=s_F$.
Similarly, we get $\widetilde{D}\big\vert_F\sim_{\mathbb{Q}} ms_F+(2m+6)f_F$.

Now we suppose that $Z$ is a~log canonical center of the~pair $(X_i,\lambda D)$.
Let us seek for a~contradiction. Since $\lambda m-1<1$ and
$K_{V}+\lambda\widetilde{D}+(\lambda m-1)F\sim_{\mathbb{Q}}\nu^*(K_{X_i}+\lambda D)$,
the surface $F$ contains an~irreducible $G$-invariant smooth rational curve $\widetilde{Z}$~such that $\nu(\widetilde{Z})=Z$,
the curve $\widetilde{Z}$ is a~section of the~projection $F\to Z$,
and $\widetilde{Z}$ is a~center of log canonical singularities the~log pair $(V,\lambda\widetilde{D}+(\lambda m-1)F)$.
Write $\widetilde{D}\vert_{F}=\theta\widetilde{Z}+\Omega$,
where $\theta$ is a~non-negative rational number,
and $\Omega$ is an~effective $\mathbb{Q}$-divisor on $F$ such that its support does not contain the~curve $\widetilde{Z}$.
Then using \cite[Theorem~5.50]{KoMo98}, we get $\theta\geqslant\frac{1}{\lambda}>\frac{4}{3}$.
On~the~other hand, we have $\widetilde{Z}\sim s_F+kf_F$ for some non-negative integer $k$ such that either $k=0$ or $k\geqslant 8$.
Thus, we have $\Omega\sim_{\mathbb{Q}} (m-\theta)s_F+(2m+6-\theta k)f_F$.
Hence, if $k\ne 0$, then
$0\leqslant 2m+6-\theta k\leqslant 2m+6-8\theta<2m+6-\frac{32}{3}=\frac{6m-14}{3}$,
so that $m>\frac{7}{3}$, which is impossible, since $m\leqslant\frac{5}{3}$.
Then $k=0$, so that $\widetilde{Z}=s_F=\widetilde{E}_i\cap F$.

Recall that $D=\varepsilon E_i+\Delta$, where $\varepsilon$ is a~non-negative rational number such that $\varepsilon\leqslant 1$,
and $\Delta$ is an~effective $\mathbb{Q}$-divisor on the~threefold $X_i$ whose support does not contain $E_i$.
Denote by $\widetilde{\Delta}$ the~proper transform of this divisor on the~threefold $V$.
Then $\widetilde{Z}$ is a~center of log canonical singularities the~log pair
$(V,\lambda\varepsilon\widetilde{E}_i+\lambda\widetilde{\Delta}+(\lambda m_{\Delta}+\lambda\varepsilon-1)F)$,
where $m_{\Delta}=\mathrm{mult}_{Z}(\Delta)$.
Using \cite[Theorem~5.50]{KoMo98} again, we see that $\widetilde{Z}$ is a~ center of log canonical singularities the~log pair
$(\widetilde{E}_i,\lambda\widetilde{\Delta}\vert_{\widetilde{E}_i}+(\lambda m_{\Delta}+\lambda\varepsilon-1)F\vert_{\widetilde{E}_i})$,
where $F\vert_{\widetilde{E}_i}=\widetilde{Z}$. This simply means that
$\lambda\widetilde{\Delta}\vert_{\widetilde{E}_i}+(\lambda m_{\Delta}+\lambda\varepsilon-1)F\vert_{\widetilde{E}_i}=c\widetilde{Z}+\Xi$
for some rational number $c\geqslant 1$, where $\Xi$ is an~effective $\mathbb{Q}$-divisor on $\widetilde{E}_i$
whose support does not contain the~curve $\widetilde{Z}$.

Now, let us compute the~numerical class of the~restriction $\widetilde{\Delta}\big\vert_{\widetilde{E}_i}$.
Observe that $\widetilde{E}_i\cong E_i$.
Denote by $s_{\widetilde{E}_i}$ and $f_{\widetilde{E}_i}$ the~strict transforms on $\widetilde{E}_i$ of the~curves $s_{E_i}$ and $f_{E_i}$, respectively.
Then
$\widetilde{\Delta}\vert_{\widetilde{E}_i}\sim_{\mathbb{Q}}(1+\varepsilon)s_{\widetilde{E}_i}+(3-\varepsilon)f_{\widetilde{E}_i}-m_{\Delta}\widetilde{Z}=(1+\varepsilon-m_{\Delta})s_{\widetilde{E}_i}+(3-\varepsilon-3m_{\Delta})f_{\widetilde{E}_i}$.
Thus, we see that
\begin{multline*}
c\big(s_{\widetilde{E}_i}+3f_{\widetilde{E}_i}\big)+\Xi\sim_{\mathbb{Q}}\lambda\widetilde{\Delta}\big\vert_{\widetilde{E}_i}+\big(\lambda m_{\Delta}+\lambda\varepsilon-1\big)F\big\vert_{\widetilde{E}_i}\sim_{\mathbb{Q}}\\
\sim_{\mathbb{Q}}\lambda(1+\varepsilon-m_{\Delta})s_{\widetilde{E}_i}+\lambda(3-\varepsilon-3m_{\Delta})f_{\widetilde{E}_i}+\big(\lambda m_{\Delta}+\lambda\varepsilon-1\big)\widetilde{Z}\sim_{\mathbb{Q}}\\
\sim_{\mathbb{Q}}(\lambda+2\lambda\varepsilon-1)s_{\widetilde{E}_i}+(3\lambda+2\lambda\varepsilon-3)f_{\widetilde{E}_i},
\end{multline*}
so that $\Xi\sim_{\mathbb{Q}} (\lambda+2\lambda\varepsilon-1-c)s_{\widetilde{E}_i}+(3\lambda+2\lambda\varepsilon-3-3c)f_{\widetilde{E}_i}$,
which gives $3\lambda+2\lambda\varepsilon-3-3c\geqslant 0$.
Since $c\geqslant 1$ and $\lambda<\frac{3}{4}$, we deduce that $\varepsilon\geqslant\frac{3}{\lambda}-\frac{3}{2}>4-\frac{3}{2}=\frac{5}{2}$.
But $\varepsilon\leqslant 1$.
The~obtained contradiction completes the~proof of the~lemma.
\end{proof}

To complete the~proof of Lemma~\ref{lemma:2-16}, we must show that  $C_i^i$ is not a~log~canonical center of the~log pair $(X_i,\lambda D)$.
Let $Z=C_i^i$. Suppose that $Z$ is a~log canonical center of the~pair $(X_i,\lambda D)$.
Let us seek for a~contradiction. Observe that $\mathrm{mult}_{Z}\big(D\big)\geqslant\frac{1}{\lambda}>\frac{4}{3}$.
Observe also that $Z$ is not a~log canonical center of the~log pair $(X_i,\lambda (F_{3,i}+E_i))$ and $D\sim_{\mathbb{Q}} F_{3,i}+E_i$.
Thus, replacing $D$ by a~divisor $(1+\mu)D-\mu(F_{3,i}+E_i)$ for an~appropriate non-negative rational number $\mu$,
we may assume that either the~surface $F_{3,i}$ or the~surface $E_i$ is not contained in the~support of the~$\mathbb{Q}$-divisor $D$.
Then we conclude that $F_{3,i}$ is not contained in the~support of the~$\mathbb{Q}$-divisor $D$, because

\begin{lemma}
\label{lemma:2-16-third}
The surface $E_i$ is contained in the~support of the~$\mathbb{Q}$-divisor $D$.
\end{lemma}

\begin{proof}
Let $\mathscr{C}$ be a~general fiber of the~projection $E_i\to Z$.
If the~surface $E_i$ is contained in the~support of the~$\mathbb{Q}$-divisor $D$,
then
$1=-K_{X_i}\cdot\mathscr{C}=D\cdot \mathscr{C}\geqslant\mathrm{mult}_{Z}(D)\geqslant\frac{1}{\lambda}>\frac{4}{3}$,
which is absurd.
\end{proof}

Let $\nu\colon V\to X_i$ be the~blow up of the~curve $Z$, let $F$ be the~$\nu$-exceptional surface,
and let $\widetilde{E}_i$ be the~strict transform of the~surface $F$ via $\nu$.
Then $F\cong\mathbb{F}_n$ for some integer $n\geqslant 0$, and
$F\vert_F\sim -s_F+af_F$ for some integer $a$, where $s_F$ is a~section of the~projection $F\to Z$ such that $s_F^2=-n$, and~$f_F$ is a~fiber of this projection.
Since $-K_{X_i}\cdot Z=4$, we conclude that $F^3=-2$. Thus, we have $-2=F^3=(-s_F+af_F)^2=-n-2a$, so that $a=\frac{2-n}{2}$. On~the~other hand, we have
$\widetilde{E}_i\big\vert_{F}\sim s_F+\frac{n-2}{2}f_F$, since $E_i\cdot Z=(-s_{E_i}+f_{E_i})\cdot(s_{E_i}+f_{E_i})=0$.
But $\widetilde{E}_i\big\vert_{F}$ is an~irreducible curve, which implies that $n=2$, since $\frac{n-2}{2}<n$.
Thus, we see that $F\cong\mathbb{F}_2$ and $-F\vert_F\sim \widetilde{E}_i\vert_{F}=s_F$.
Observe also that the~action of the~group $G$ lifts to the~threefold $V$, since $Z$ is $G$-invariant.

\begin{remark}
\label{remark:2-16-KV-nef}
The divisor $-K_V$ is nef and big.
Indeed, the~linear system $|\pi_i^*(2H)-2E_i|$ is base point free.
Let $\mathcal{M}$ be its strict transform on $V$.
Then $\mathcal{M}+\widetilde{E}_i$ is a~linear subsystem of the~linear system $|-K_V|$,
so that the~base locus of the~linear system $|-K_V|$ is contained in $\widetilde{E}_i$.
But $\widetilde{E}_i\cong E_i$ and $-K_V\vert_{\widetilde{E}_i}\sim 2f_{\widetilde{E}_i}$,
where $f_{\widetilde{E}_i}$ is a~strict transform of the~curve $f_{E_i}$ on the~surface $\widetilde{E}_i$.
Then  $-K_V\vert_{\widetilde{E}_i}$ is nef, so that $-K_V$ is also nef.
Since $-K_V^3=12$, we see that  $-K_V$ is big.
\end{remark}

Let $m=\mathrm{mult}_{Z}(D)$, and let $\widetilde{D}$ be the~proper transform of the~divisor $D$ via $\nu$.
Then
$$
\widetilde{D}\big\vert_F\sim_{\mathbb{Q}} \big(\nu^*(-K_{X_i})-mF\big)\big\vert_F\sim_{\mathbb{Q}} ms_F+4f_F.
$$
Let $\mathscr{C}$ be a~sufficiently general fiber of the~conic bundle $\nu_i$ that is contained in $F_{3,i}$,
and let~$\widetilde{\mathscr{C}}$ be its strict transform on the~threefold $V$.
Then $\mathscr{C}$ is an~irreducible curve that is not contained in the~support of the~divisor $D$,
because we assumed that $F_{3,i}\not\subset\mathrm{Supp}(D)$.
Moreover, the~curve $\mathscr{C}$ intersects the~curve $Z$, because $F_{3,i}\vert_{E_i}=2Z$.
Thus, we have
$$
2-m=2-mF\cdot\widetilde{\mathscr{C}}=\big(\nu^*(-K_{X_i})-mF\big)\cdot\widetilde{\mathscr{C}}=\widetilde{D}\cdot\widetilde{\mathscr{C}}\geqslant 0,
$$
so that $m\leqslant 2$.
Since $\lambda m-1<1$ and $K_{V}+\lambda\widetilde{D}+(\lambda m-1)F\sim_{\mathbb{Q}}\nu^*(K_{X_i}+\lambda D)$,
the surface $F$ contains an~irreducible $G$-invariant smooth curve $\widetilde{Z}$~such that $\nu(\widetilde{Z})=Z$,
the curve $\widetilde{Z}$ is a~section of the~projection $F\to Z$,
and $\widetilde{Z}$ is a~center of log canonical singularities the~log pair $(V,\lambda\widetilde{D}+(\lambda m-1)F)$.
Let $\widetilde{m}=\mathrm{mult}_{\widetilde{Z}}(\widetilde{D})$. Then
\begin{equation}
\label{equation:2-16}
m+\widetilde{m}\geqslant\frac{2}{\lambda}>\frac{8}{3},
\end{equation}
because the~multiplicity of the~divisor $\lambda\widetilde{D}+(\lambda m-1)F$ at the~curve $\widetilde{Z}$ must be at least $1$.

\begin{lemma}
\label{lemma:2-16-three-and-half}
Either $\widetilde{Z}=s_F$ or $\widetilde{Z}\sim s_F+2f_F$.
\end{lemma}

\begin{proof}
Write $\widetilde{D}\vert_{F}=\theta\widetilde{Z}+\Omega$, where $\theta$ is a~non-negative rational number,
and $\Omega$ is an~effective $\mathbb{Q}$-divisor on $F$ such that its support does not contain $\widetilde{Z}$.
Using \cite[Theorem~5.50]{KoMo98}, we get $\theta\geqslant\frac{1}{\lambda}>\frac{4}{3}$.
But $\widetilde{Z}\sim s_F+kf_F$ for $k\in\mathbb{Z}$ such that $k=0$ or $k\geqslant 2$.
Thus, we have
$$
\Omega\sim_{\mathbb{Q}} ms_F+4f_F-\theta\widetilde{Z}\sim_{\mathbb{Q}}(m-\theta)s_F+(4-\theta k)f_F.
$$
Hence, if $k\ne 0$, then $0\leqslant 4-\theta k<4-\frac{4}{3}k$, so that $k=2$.
Then $\widetilde{Z}=s_F$ or $\widetilde{Z}\sim s_F+2f_F$.
\end{proof}

Let $\widetilde{F}_{3,i}$ be the~proper transform on $V$ of the~surface $F_{3,i}$.
If $\widetilde{Z}=s_F$, then $\widetilde{Z}=\widetilde{E}_i\cap\widetilde{F}_{3,i}$,
because $F_{3,i}$ is tangent to $E_i$ along the~curve $Z$ and $\widetilde{E}_i\cap\vert_{F}=\widetilde{Z}$.
Using this, we get

\begin{lemma}
\label{lemma:2-16-forth}
One has $\widetilde{Z}\ne s_F$.
\end{lemma}

\begin{proof}
If $\widetilde{Z}=s_F$, then $\widetilde{\mathscr{C}}$ intersects the~curve $\widetilde{Z}$, so hat
$2-m\geqslant 2-mF\cdot\widetilde{\mathscr{C}}=\widetilde{D}\cdot\widetilde{\mathscr{C}}\geqslant\widetilde{m}$,
which contradicts \eqref{equation:2-16}.
\end{proof}

Thus, we see that $\widetilde{Z}\sim s_F+2f_F$.

\begin{remark}
\label{remark:2-16-unique-curve-F}
The curve $\widetilde{Z}$ is unique $G$-invariant curve in the~linear system $|s_F+2f_F|$,
because $(s_F+2f_F)\cdot\widetilde{Z}=2$, and $\widetilde{Z}$ does not have $G$-orbits of length less than $4$.
\end{remark}

Let $\rho\colon Y\to V$ be the~blow up of the~curve $\widetilde{Z}$,
and let $R$ be the~$\rho$-exceptional surface. Then $-K_{Y}^3=2$.

\begin{lemma}
\label{lemma:2-16-fifth}
The divisor $-K_Y$ is nef.
\end{lemma}

\begin{proof}
Let $\widehat{F}_{3,i}$, $\widehat{E}_i$, $\widehat{F}$ be the~strict transforms of the~surfaces $F_{3,i}$, $E_i$, $F$, respectively.
Then $|-K_Y|$ contains the~divisor $\widehat{F}_{3,i}+\widehat{E}_i+\widehat{F}$.
Therefore, to prove the~required assertion, it is enough to prove that the~restrictions
$-K_Y\vert_{\widehat{F}_{3,i}}$, $-K_Y\vert_{\widehat{E}_i}$ and $-K_Y\vert_{\widehat{F}}$ are nef.

The nefness of the~restriction $-K_Y\vert_{\widehat{E}_i}$ follows from the~nefness of the~restriction $-K_V\vert_{\widetilde{E}_i}$,
because $\widetilde{Z}$ is disjoint from the~surface $\widetilde{E}_i$.
To check the~nefness of the~restriction $-K_Y\vert_{\widehat{F}}$, note that $\widetilde{Z}\sim s_F+2f_F$ and
$-K_V\vert_{F}\sim s_F+4f_F$,
so that $-K_Y\vert_{\widehat{F}}$ is rationally equivalent to the~sum of two fibers of the~ projection $\widehat{F}\to\mathbb{P}^1$.
Hence, the~restriction $-K_Y\vert_{\widehat{F}}$ is nef.

Thus, we must prove that  $-K_Y\vert_{\widehat{F}_{3,i}}$ is nef.
To do this, recall that $F_{3,i}$ is a~preimage via the~conic bundle $\eta_i$ of a~$G$-invariant conic in $\mathbb{P}^2$,
which we denoted earlier by $\mathcal{C}_3$.
Using explicit equation of the~surface $F_{3,i}$, one can check that
this conic intersects the~discriminant curve $\Delta_i$ by four points that form a~$G$-orbit of length $4$,
so that $\mathcal{C}_3$ has simple tangency with $\Delta_i$ at every intersection point.
Denote the~points in $\mathcal{C}_3\cap\Delta_i$ by $P_1$, $P_2$, $P_3$ and $P_4$.
For each $k\in\{1,2,3,4\}$, we have $\eta_i^{-1}(P_k)=\ell_k+\ell_k^\prime$,
where $\ell_k$ and $\ell_k^\prime$ are smooth rational curve that intersect transversally at one point.
Thus, in total we obtain eight smooth rational curves $\ell_1$, $\ell_1^\prime$, $\ell_2$, $\ell_2^\prime$, $\ell_3$, $\ell_3^\prime$, $\ell_4$, $\ell_4^\prime$.
Denote their images in $V_4$ by $\overline{\ell}_1$, $\overline{\ell}_1^\prime$, $\overline{\ell}_2$, $\overline{\ell}_2^\prime$,
$\overline{\ell}_3$, $\overline{\ell}_3^\prime$, $\overline{\ell}_4$, $\overline{\ell}_4^\prime$, respectively.
Then these eight curves are lines, which we will describe later.
Similarly, denote their strict transforms on $V$ by $\widetilde{\ell}_1$, $\widetilde{\ell}_1^\prime$, $\widetilde{\ell}_2$, $\widetilde{\ell}_2^\prime$, $\widetilde{\ell}_3$, $\widetilde{\ell}_3^\prime$, $\widetilde{\ell}_4$, $\widetilde{\ell}_4^\prime$, respectively.
Then, by construction, we have
$$
-K_{V}\cdot\widetilde{\ell}_1=-K_{V}\cdot\widetilde{\ell}_1^\prime=-K_{V}\cdot\widetilde{\ell}_2=-K_{V}\cdot\widetilde{\ell}_2^\prime=-K_{V}\cdot\widetilde{\ell}_3=-K_{V}\cdot\widetilde{\ell}_3^\prime=-K_{V}\cdot\widetilde{\ell}_4=-K_{V}\cdot\widetilde{\ell}_4^\prime=0.
$$
Finally, let us denote the~strict transforms on $Y$ of these eight curves  by $\widehat{\ell}_1$, $\widehat{\ell}_1^\prime$, $\widehat{\ell}_2$, $\widehat{\ell}_2^\prime$, $\widehat{\ell}_3$, $\widehat{\ell}_3^\prime$, $\widehat{\ell}_4$, $\widehat{\ell}_4^\prime$, respectively.
For every $k\in\{1,2,3,4\}$, we have
$-K_{Y}\cdot\widehat{\ell}_k=-R\cdot\widehat{\ell}_k$ and $-K_{Y}\cdot\widehat{\ell}_k^\prime=-R\cdot\widehat{\ell}_k^\prime$.
Therefore, if $\widetilde{Z}$ intersects a~curve $\widehat{\ell}_k$ or $\widehat{\ell}_k^\prime$,
then $-K_Y$ is not nef, because in these case we have $-K_{Y}\cdot\widehat{\ell}_k<0$ or $-K_{Y}\cdot\widehat{\ell}_k^\prime<0$, respectively.

First, let us show that the~curves $\widehat{\ell}_1$, $\widehat{\ell}_1^\prime$, $\widehat{\ell}_2$, $\widehat{\ell}_2^\prime$, $\widehat{\ell}_3$, $\widehat{\ell}_3^\prime$, $\widehat{\ell}_4$, $\widehat{\ell}_4^\prime$
are the~only curves in $\widehat{F}_{3,i}$ that a~priori may have negative intersections with the~divisor $-K_Y$.
After thus, we will explicitly check that $\widetilde{Z}$ does not intersects any of the~curves
$\widetilde{\ell}_1$, $\widetilde{\ell}_1^\prime$, $\widetilde{\ell}_2$, $\widetilde{\ell}_2^\prime$, $\widetilde{\ell}_3$, $\widetilde{\ell}_3^\prime$, $\widetilde{\ell}_4$, $\widetilde{\ell}_4^\prime$,
which would imply that $-K_Y$ is indeed nef.

By construction, the~curves $\ell_1$, $\ell_1^\prime$, $\ell_2$, $\ell_2^\prime$, $\ell_3$, $\ell_3^\prime$, $\ell_4$, $\ell_4^\prime$
form two $G$-irreducible curves each consisting of four irreducible components.
Without loss of generality, we may assume that $\ell_1+\ell_2+\ell_3+\ell_4$ is one of these curves,
and $\ell_1^\prime+\ell_2^\prime+\ell_3^\prime+\ell_4^\prime$ is another curve.

Observe that $\widetilde{F}_{3,i}\vert_{F}\sim s_F+4f_F$ and the~intersection $\widetilde{F}_{3,i}\cap F$ contains the~curve $s_F$.
This implies that $\widetilde{F}_{3,i}\vert_{F}=s_F+e_1+e_2+e_3+e_4$,
where $e_k$ is a~fiber of the~projection $F\to Z$ such that $\nu(e_k)=\ell_k\cap\ell_k^\prime$.
Since $\widetilde{F}_{3,i}\vert_{\widetilde{E}_i}=s_F$, we see that $\widetilde{F}_{3,i}$ is smooth.
Moreover, we have $(s_F\cdot s_F)_{\widetilde{F}_{3,i}}=-2$, because $\widetilde{E}_i^2\cdot \widetilde{F}_{3,i}=-2$.
Now, using this and $F^2\cdot \widetilde{F}_{3,i}=-2$, we conclude that
$(e_1\cdot e_1)_{\widetilde{F}_{3,i}}=(e_2\cdot e_2)_{\widetilde{F}_{3,i}}=(e_3\cdot e_3)_{\widetilde{F}_{3,i}}=(e_4\cdot e_5)_{\widetilde{F}_{3,i}}=-2$.
Thus, we conclude that $F_{3,i}$ has an~ordinary double point at each point $\ell_k\cap\ell_k^\prime$,
and the~birational morphism $\nu$ induces the~minimal resolution of singularities $\widetilde{F}_{3,i}\to F_{3,i}$,
which contracts the~curve $e_k$ to the~point $\ell_k\cap\ell_k^\prime$.

The composition $\eta_i\circ\nu$ induces a~conic bundle $\widetilde{F}_{3,i}\to\mathcal{C}_3$.
The curve $s_F$ is its section, and its (scheme) fibers over the~points $P_1$, $P_2$, $P_3$, $P_4$ are
$e_1+\widetilde{\ell}_1+\widetilde{\ell}_1^\prime$,
$e_2+\widetilde{\ell}_1+\widetilde{\ell}_2^\prime$,
$e_3+\widetilde{\ell}_1+\widetilde{\ell}_3^\prime$,
$e_4+\widetilde{\ell}_1+\widetilde{\ell}_4^\prime$,
respectively.
Thus, for every $k\in\{1,2,3,4\}$, the~curves  $\widetilde{\ell}_k$ and $\widetilde{\ell}_k^\prime$ are disjoint $(-1)$-curves on the~surface $\widetilde{F}_{3,i}$,
which both do not intersect the~section $s_F$, because $s_F$ intersects the~$(-2)$-curve $e_k$.
Moreover, we have
$$
-K_V\big\vert_{\widetilde{F}_{3,i}}\sim s_F+\sum_{k=1}^{4}\big(e_k+\widetilde{\ell}_k+\widetilde{\ell}_k^\prime\big),
$$
because $-K_V\sim \nu^*(F_{3,i})+\widetilde{E}_i$ and $\widetilde{E}_i\vert_{\widetilde{F}_{3,i}}=s_F$.

The curve $\widetilde{Z}$ intersects the~surface $\widetilde{F}_{3,i}$ transversally by a~$G$-orbit of length~$4$,
because it intersects the~(reducible) curve $s_F+e_1+e_2+e_3+e_4$ transversally by the~points
$\widetilde{Z}\cap e_1$, $\widetilde{Z}\cap e_2$, $\widetilde{Z}\cap e_3$, $\widetilde{Z}\cap e_4$,
which form one $G$-orbit.
Thus, the~morphism $\rho$ induces a~birational morphism $\varrho\colon\widehat{F}_{3,i}\to\widetilde{F}_{3,i}$ that is a~a blow up of this $G$-orbit.
Using this, we see that
$$
-K_Y\big\vert_{\widehat{F}_{3,i}}\sim \varrho^*\Big(s_F+\sum_{k=1}^{4}\big(e_k+\widetilde{\ell}_k+\widetilde{\ell}_k^\prime\big)\Big)-r_1-r_2-r_3-r_4
$$
where $r_k$ is the~exceptional curve of $\varrho$ that is contracted to the~point $\widetilde{Z}\cap e_k$.
Observe that these four points $\widetilde{Z}\cap e_1$, $\widetilde{Z}\cap e_2$, $\widetilde{Z}\cap e_3$, $\widetilde{Z}\cap e_4$
are not contained in the~curve $s_F$, because the~curves $\widetilde{Z}$ and $s_F$ are disjoint.
Moreover, we have three mutually excluding options:
\begin{enumerate}
\item  the~$G$-orbit $\widetilde{Z}\cap\widetilde{F}_{3,i}$ is contained in the~curve $\widetilde{\ell}_1+\widetilde{\ell}_2+\widetilde{\ell}_3+\widetilde{\ell}_4$;
\item  the~$G$-orbit $\widetilde{Z}\cap\widetilde{F}_{3,i}$ is contained in the~curve $\widetilde{\ell}_1^\prime+\widetilde{\ell}_2^\prime+\widetilde{\ell}_3^\prime+\widetilde{\ell}_4^\prime$;
\item the~$G$-orbit $\widetilde{Z}\cap\widetilde{F}_{3,i}$ is contained in the~curves $\widetilde{\ell}_1+\widetilde{\ell}_2+\widetilde{\ell}_3+\widetilde{\ell}_4$ and $\widetilde{\ell}_1^\prime+\widetilde{\ell}_2^\prime+\widetilde{\ell}_3^\prime+\widetilde{\ell}_4^\prime$.
\end{enumerate}
As we already mentioned, the~divisor $-K_Y$ is not nef in the~first two cases.
In the~third case, we have
$$
-K_Y\big\vert_{\widehat{F}_{3,i}}\sim\widehat{s}_F+\sum_{k=1}^{4}\big(\widehat{e}_k+\widehat{\ell}_k+\widehat{\ell}_k^\prime\big),
$$
where $\widehat{s}_F$ and $\widehat{e}_k$ are strict transforms of the~curves $s_F$ and $e_k$ on the~surface $\widehat{F}_{3,i}$.
Moreover, in this case, we have
$\widehat{s}_F\cdot\widehat{s}_F=-2$, $\widehat{s}_F\cdot \widehat{e}_k=1$, $\widehat{e}_k=1\cdot\widehat{e}_k=-1$, $\widehat{e}_k\cdot \widehat{e}_k=-3$,
$\widehat{e}_k\cdot\widehat{\ell}_k=1$, $\widehat{e}_k\cdot\widehat{\ell}_k^\prime=1$ on the~surface $\widehat{F}_{3,i}$, and all other intersections are zero.
This immediately implies that the~divisor $-K_Y\big\vert_{\widehat{F}_{3,i}}$ is nef in the~third case,
so that $-K_Y$ is also nef.

Therefore, we proved that the~divisor $-K_Y$ is nef if and only if the~curve $\widetilde{Z}$ does not intersect the~curves $\widetilde{\ell}_1+\widetilde{\ell}_2+\widetilde{\ell}_3+\widetilde{\ell}_4$,
and $\widetilde{\ell}_1+\widetilde{\ell}_2+\widetilde{\ell}_3+\widetilde{\ell}_4$.
Observe that these curves intersects the~$\nu$-exceptional surface $F$ by two (distinct) $G$-orbits of length $4$, respectively.
Denote these $G$-orbits by $\Theta$ and $\Theta^\prime$, respectively.
Hence, to complete the~proof, it is enough to check that neither $\Theta$ nor $\Theta^\prime$ is contained in the~curve $\widetilde{Z}$.

We have $h^0(\mathcal{O}_{V}(-K_V))=9$ by the~Riemann--Roch formula and the~Kawamata--Viehweg vanishing,
since $-K_V$ is big and nef by Remark~\ref{remark:2-16-KV-nef}.
Moreover, we have $-K_V\vert_{F}\sim s_F+4f_F$ and
$h^0(\mathcal{O}_{F}(s_F+4f_F))=8$.
Furthermore, we have $h^0(\mathcal{O}_{V}(-K_V-F))=1$, since the~linear system $|-K_V-F|$ contains unique effective divisor: $\widetilde{F}_{3,i}+\widetilde{E}_i$.
This gives the~following exact sequence of $G$-representations:
\begin{equation}
\label{equation:2-16-exact-sequence}
0\longrightarrow H^0\Big(\mathcal{O}_{V}\big(\widetilde{F}_{3,i}+\widetilde{E}_i\big)\Big)\longrightarrow H^0\Big(\mathcal{O}_{V}\big(-K_V\big)\Big)\longrightarrow H^0\Big(\mathcal{O}_{F}\big(s_F+4f_F\big)\Big)\longrightarrow 0.
\end{equation}
Here, the~kernel of the~third map is the~one-dimensional $G$-representation generated by the~section
vanishing on the~divisor $\widetilde{F}_{3,i}+\widetilde{E}_i+F$.

Note that $s_F\cong\mathbb{P}^1$ and $(s_F+4f_F)\cdot s_F=2$.
Thus, the Riemann--Roch formula and the~Kawamata--Viehweg vanishing give
the following exact sequence of $G$-representations:
$$
0\longrightarrow H^0\big(\mathcal{O}_{F}(4f_F)\big)\longrightarrow H^0\big(\mathcal{O}_{F}(s_F+4f_F)\big)\longrightarrow H^0\big(\mathcal{O}_{\mathbb{P}^1}(2)\big)\longrightarrow 0.
$$
Since $s_F$ does not have $G$-orbits of length $2$,
we have $H^0(\mathcal{O}_{\mathbb{P}^1}(2))\cong\mathbb{U}_3$, where $\mathbb{U}_3$ is the~unique irreducible three-dimensional representation of the~group $G$.
Similarly, since $Z$ has exactly two $G$-orbits of length $4$,
we have $H^0(\mathcal{O}_{F}(4f_F))\cong\mathbb{U}_1\oplus\mathbb{U}_1^\prime\oplus\mathbb{U}_3$,
where $\mathbb{U}_1$ and $\mathbb{U}_1^\prime$ are different one-dimensional representations of the~group $G$.
Thus, one has
$$
H^0\Big(\mathcal{O}_{F}\big(s_F+4f_F\big)\Big)\cong\mathbb{U}_1\oplus\mathbb{U}_1^\prime\oplus\mathbb{U}_3\oplus\mathbb{U}_3.
$$
We may assume that $\mathbb{U}_1$ is generated by a~section that vanishes at $s_F+e_1+e_2+e_3+e_4$.

Let $\mathbb{V}$ and $\mathbb{V}^\prime$ be sub-representations in $H^0(\mathcal{O}_{F}(s_F+4f_F))$
that consist of all sections vanishing at the~$G$-orbits $\Theta$ and $\Theta^\prime$, respectively.
Then $\mathrm{dim}(\mathbb{V})=\mathrm{dim}(\mathbb{V}^\prime)=4$, so that
$$
\mathbb{V}\cong\mathbb{V}^\prime\cong \mathbb{U}_1\oplus\mathbb{U}_3,
$$
since both $G$-orbits $\Theta$ and $\Theta^\prime$ are contained in $s_F+e_1+e_2+e_3+e_4$ by construction.
Let~$\widetilde{\mathbb{V}}$ and $\widetilde{\mathbb{V}}^\prime$ be the~the preimages
in $H^0(\mathcal{O}_{V}(-K_V))$ via the~restriction map in \eqref{equation:2-16-exact-sequence} of the~sub-representations $\mathbb{V}$ and $\mathbb{V}^\prime$, respectively.
Then, as $G$-representations, we have
$$
\widetilde{\mathbb{V}}\cong\widetilde{\mathbb{V}}^\prime\cong \mathbb{U}_1\oplus\mathbb{U}_1^{\prime\prime}\oplus\mathbb{U}_3,
$$
where $\mathbb{U}_1^{\prime\prime}$ is a~one-dimensional representation of the~group $G$.
Since $\widetilde{\mathbb{V}}$ and $\widetilde{\mathbb{V}}^\prime$ contain unique three-dimensional subrepresentation of the~group $G$,
these (two) three-dimensional subrepresentations define two $G$-invariant linear subsystems $\mathcal{M}_V$ and $\mathcal{M}_V^\prime$ of the~linear system $|-K_V|$, respectively.
They can be characterized as (unique) three-dimensional $G$-invariant linear subsystems in $|-K_V|$ that contains $G$-orbits $\Theta$ and $\Theta^\prime$, respectively.
Then $\mathcal{M}_V\vert_{F}$ and $\mathcal{M}_V^\prime\vert_{F}$ are
(unique)  three-dimensional $G$-invariant linear subsystems of the~linear system $|s_F+4f_F|$
that contain $\Theta$ and $\Theta^\prime$, respectively.
Thus, if $\Theta\subset\widetilde{Z}$, then
$$
\mathcal{M}_V\big\vert_{F}=\widetilde{Z}+|2f_F|,
$$
so that $\widetilde{Z}\subseteq\mathrm{Bs}(\mathcal{M}_V)$.
Similarly, if $\Theta^\prime\subset\widetilde{Z}$, then $\mathcal{M}_V^\prime\vert_{F}=\widetilde{Z}+|2f_F|$,
so that $\widetilde{Z}\subseteq\mathrm{Bs}(\mathcal{M}_V^\prime)$.

Let $\mathcal{M}$ and $\mathcal{M}^\prime$ be strict transforms on $V_4$ of the~linear systems $\mathcal{M}_V$ and $\mathcal{M}_V^\prime$, respectively.
Then $\mathcal{M}$ and $\mathcal{M}^\prime$ are linear subsystems in $|2H|$, so that they do not have fixed components,
because $|H|$ does not have $G$-invariant divisors.
Let $M_1$ and $M_2$ be two distinct surfaces in $\mathcal{M}$.
If $\Theta\subset\widetilde{Z}$, then
\begin{equation}
\label{equation:2-16-M1-M2}
\Big(M_1\cdot M_2\Big)_{C_i}\geqslant 3.
\end{equation}
Similarly, if $\Theta^\prime\subset\widetilde{Z}$, then
\begin{equation}
\label{equation:2-16-M1-M2-prime}
\Big(M_1^\prime\cdot M_2^\prime\Big)_{C_i}\geqslant 3,
\end{equation}
where $M_1^\prime$ and $M_2^\prime$ are two surfaces in $\mathcal{M}^\prime$.
Both conditions \eqref{equation:2-16-M1-M2} and \eqref{equation:2-16-M1-M2-prime} are easy to check
provided that we know  generators of the~linear system $\mathcal{M}$ and $\mathcal{M}^\prime$.

Observe that the~curve $\widetilde{\ell}_1+\widetilde{\ell}_2+\widetilde{\ell}_3+\widetilde{\ell}_4$
is contained in the~base locus of the~linear system $\mathcal{M}_V$.
Indeed, one has $\mathcal{M}_V\subset|-K_V|$ and $-K_V\cdot \widetilde{\ell}_i=0$ for every $i\in\{1,2,3,4\}$,
while $\Theta\subseteq\mathrm{Bs}(\mathcal{M}_V)$ by construction,
and $\Theta$ is contained in  $\widetilde{\ell}_1+\widetilde{\ell}_2+\widetilde{\ell}_3+\widetilde{\ell}_4$ by definition.
Likewise, we see that $\widetilde{\ell}_1^\prime+\widetilde{\ell}_2^\prime+\widetilde{\ell}_3^\prime+\widetilde{\ell}_4^\prime$
is contained in the~base locus of the~linear system~$\mathcal{M}_V^\prime$.
Hence, the~$G$-irreducible curves $\overline{\ell}_1+\overline{\ell}_2+\overline{\ell}_3+\overline{\ell}_4$
and $\overline{\ell}_1^\prime+\overline{\ell}_2^\prime+\overline{\ell}_3^\prime+\overline{\ell}_4^\prime$ are
contained in the~base loci of the~linear systems $\mathcal{M}$ and $\mathcal{M}^\prime$, respectively.
Moreover, the~base loci of these linear systems also contain the~conic $C_i$.
Using these linear conditions, we can find the~generators of these linear systems,
and check the~conditions \eqref{equation:2-16-M1-M2} and \eqref{equation:2-16-M1-M2-prime}.

Since $X_1\cong X_2$ and $X_3\cong X_4$, it is enough to consider only the~cases $i=1$ and $i=3$.
First, we deal with the~case $i=1$. In this case, the~curves
$\overline{\ell}_1+\overline{\ell}_2+\overline{\ell}_3+\overline{\ell}_4$
and $\overline{\ell}_1^\prime+\overline{\ell}_2^\prime+\overline{\ell}_3^\prime+\overline{\ell}_4^\prime$
can be described as follows: up to a~swap and a~reshuffle, we may assume that
\begin{itemize}
\item $\overline{\ell}_1$ is the~line $[\lambda:\omega \lambda:-(\omega+1)\lambda:\mu-(\omega+2)\lambda:\mu:\mu+(\omega-1)\lambda]$,
\item $\overline{\ell}_2$ is the~line $[\lambda:-\omega \lambda:-(\omega+1)\lambda:-\mu-(\omega+2)\lambda:\mu:-\mu+(\omega-1)\lambda]$,
\item $\overline{\ell}_3$ is the~line $[\lambda:\omega \lambda:(\omega+1)\lambda:\mu-(\omega+2)\lambda:\mu:-\mu+(-\omega+1)\lambda]$,
\item $\overline{\ell}_4$ is the~line $[\lambda:-\omega \lambda:(\omega+1)\lambda:-\mu-(\omega+2)\lambda:\mu:\mu+(-\omega+1)\lambda]$,
\end{itemize}
and
\begin{itemize}
\item $\overline{\ell}_1^\prime$ is the~line $[\lambda:\omega \lambda:-(\omega+1)\lambda:\mu+(2\omega+1)\lambda:\mu:\mu+(\omega+2)\lambda]$,
\item $\overline{\ell}_2^\prime$ is the~line $[\lambda:-\omega \lambda:-(\omega+1)\lambda:-\mu+(2\omega+1)\lambda:\mu:-\mu+(\omega+2)\lambda]$,
\item $\overline{\ell}_3^\prime$ is the~line $[\lambda:\omega \lambda:(\omega+1)\lambda:\mu+(2\omega+1)\lambda:\mu:-\mu-(\omega+2)\lambda]$,
\item $\overline{\ell}_4^\prime$ is the~line $[\lambda:-\omega \lambda:(\omega+1)\lambda:-\mu+(2\omega+1)\lambda:\mu:\mu-(\omega+2)\lambda]$,
\end{itemize}
where $[\lambda:\mu]\in\mathbb{P}^1$.
Therefore, the~linear subsystem in $|2H|$ that consists of all surfaces containing the~conic  $C_1$
and the~curve $\overline{\ell}_1+\overline{\ell}_2+\overline{\ell}_3+\overline{\ell}_4$ is five-dimensional.
Moreover, it is generated by
the $G$-invariant surfaces $\overline{F}_{1,1}$, $\overline{F}_{3,1}$, $\overline{F}_{2,3}$, and the~$G$-invariant two-dimensional
linear subsystem (net)  that is cut out on $V_4$ by
\begin{multline}
\label{equation:2-16-M-case-1}
\lambda\Big((1-\omega )x_0x_5-(2\omega+1)x_2x_3+3x_0x_2\Big)+\\
+\mu\Big((\omega+1)x_1x_3+(2\omega+1)x_0x_1+x_4x_0\Big)+\\
+\gamma\Big((\omega+2)x_1x_2-\omega x_1x_5+x_2x_4\Big)=0,
\end{multline}
where $[\lambda:\mu:\gamma]\in\mathbb{P}^2$.
Therefore, we conclude that \eqref{equation:2-16-M-case-1} defines the~linear system~$\mathcal{M}$.
It follows from \eqref{equation:2-16-M-case-1} that the~base locus of this linear system
consists of the~conic $C_1$, the~curve $\overline{\ell}_1+\overline{\ell}_2+\overline{\ell}_3+\overline{\ell}_4$, and the~conic $C_3$.
Similarly, we see that $\mathcal{M}^\prime$ is given by
\begin{multline*}
\lambda\Big((2\omega+1)x_0x_5+(\omega+2)x_2x_3+3x_0x_2\Big)+\\
+\mu\Big((\omega+1)x_1x_3+(1-\omega)x_0x_1+x_4x_0 \Big)+\\
+\gamma\Big((2\omega+1)x_1x_2+\omega x_1x_5-x_2x_4\Big)=0,
\end{multline*}
where $[\lambda:\mu:\gamma]\in\mathbb{P}^2$.
We also see that the~base locus of the~linear system $\mathcal{M}^\prime$ consists of the~conic~$C_1$, the~curve $\overline{\ell}_1^\prime+\overline{\ell}_2^\prime+\overline{\ell}_3^\prime+\overline{\ell}_4^\prime$,
and the~conic $C_4$.
Now one can check that neither \eqref{equation:2-16-M1-M2} nor \eqref{equation:2-16-M1-M2-prime} holds.
Thus, if $i=1$ or $i=2$, then $-K_Y$ is nef.

Finally, we consider the~case $i=3$. Now, up to a~swap, the~linear system $\mathcal{M}$ is again given by \eqref{equation:2-16-M-case-1},
and the~linear system $\mathcal{M}^\prime$ is given by
\begin{multline*}
\lambda\Big((\omega+1)x_1x_5+x_4x_2-(\omega-1)x_4x_5\Big)+\\
+\mu\Big(\omega x_0x_5-x_3x_2+(2\omega+1)x_3x_5\Big)+\\
+\gamma\Big(\omega x_0x_5-x_3x_2+(2\omega+1)x_3x_5\Big)=0,
\end{multline*}
where $[\lambda:\mu:\gamma]\in\mathbb{P}^2$.
Note that the~base locus of the~net $\mathcal{M}^\prime$
consists of the~conic~$C_3$, the~curve $\overline{\ell}_1^\prime+\overline{\ell}_2^\prime+\overline{\ell}_3^\prime+\overline{\ell}_4^\prime$,
and the~conic $C_2$.
As above, one can check that neither \eqref{equation:2-16-M1-M2} nor \eqref{equation:2-16-M1-M2-prime} holds.
Thus, the~divisor $-K_Y$ is nef.
\end{proof}

Let $\widehat{D}$ be the~proper transform of the~divisor $D$ on the~threefold $Y$. Then
$$
\widehat{D}\sim_{\mathbb{Q}}\sim (\pi_i\circ\nu\circ\rho)^*(2H)-(\nu\circ\rho)^*(E_i)-m\rho^*(F)-\widetilde{m}R.
$$
Since $-K_Y$ is nef, we see that $-K_Y^2\cdot\widehat{D}\geqslant 0$.
To compute  $-K_Y^2\cdot\widehat{D}$, observe that
\begin{multline*}
H^3=4,\pi_i^*(H)\cdot E^2=-2, (\pi_i\circ\nu)^*(H)\cdot F^2=-2, \\
(\pi_i\circ\nu\circ\rho)^*(H)\cdot R^2=-2, E^3=-2, F^3=-2, R^3=-2,
\end{multline*}
and other intersections involved in the~computation $-K_Y^2\cdot\widehat{D}$ are all zero.
This gives
$$
0\leqslant -K_Y^2\cdot\widehat{D}=\Big((\pi_i\circ\nu\circ\rho)^*(2H)-(\nu\circ\rho)^*(E_i)-\rho^*(F)-R\Big)^2\cdot \widehat{D}=14-6(m+\widetilde{m}),
$$
so that $m+\widetilde{m}\leqslant\frac{7}{3}$, which is impossible by \eqref{equation:2-16}.
The obtained contradiction completes the~proof of Lemma~\ref{lemma:2-16},
which completes the~proof of Proposition~\ref{proposition:2-16}.
Thus, we see that the~threefolds $X_1$, $X_2$, $X_3$ and $X_4$ are K-polystable.

\end{document}